\documentclass[11pt]{amsart}

\usepackage[a4paper,hmargin=3.5cm,vmargin=4cm]{geometry}
\usepackage{amsfonts,amssymb,amscd,amstext}
\usepackage{graphicx}
\usepackage[dvips]{epsfig}

\usepackage{fancyhdr}
\input xy
\xyoption{all}

\usepackage{times}

\usepackage{enumerate}
\usepackage{titlesec}
\usepackage{mathrsfs}

\titleformat{\section}
{\filcenter\bfseries\large} {\thesection{.}}{0.2cm}{}
\titleformat{\subsection}[runin]
{\bfseries} {\thesubsection{.}}{0.15cm}{}[.]
\titleformat{\subsubsection}[runin]
{\em}{\thesubsubsection{.}}{0.15cm}{}[.]

\usepackage[up,bf]{caption}

\newtheorem{theorem}{Theorem}[section]
\newtheorem{proposition}[theorem]{Proposition}
\newtheorem{claim}[theorem]{Claim}
\newtheorem{lemma}[theorem]{Lemma}
\newtheorem{corollary}[theorem]{Corollary}

\theoremstyle{definition}
\newtheorem{definition}[theorem]{Definition}
\newtheorem{remark}[theorem]{Remark}

\newtheorem{problem}[theorem]{Problem}

\numberwithin{equation}{section}
\numberwithin{figure}{section}

\def\Ccal{\mathcal{C}}

\def\Lcal{\mathcal{L}}
\def\Pcal{\mathcal{P}}

\def\be{\mathbf{e}}

\def\Ascr{\mathscr{A}}
\def\Cscr{\mathscr{C}}

\def\Oscr{\mathscr{O}}

\def\c{\mathbb{C}}
\def\cp{\mathbb{CP}}
\def\z{\mathbb{Z}}

\def\r{\mathbb{R}}
\def\n{\mathbb{N}}

\def\z{\mathbb{Z}}

\def\pgot{\mathfrak{p}}

\def\Agot{\mathfrak{A}}
\def\Igot{\mathfrak{I}}

\def\Ogot{\mathfrak{O}}

\def\dist{\mathrm{dist}}

\def\length{\mathrm{length}}
\def\Flux{\mathrm{Flux}}

\newcommand\wt{\widetilde}
\newcommand\wh{\widehat}

\newcommand\di{\partial}
\newcommand\dibar{\overline\partial}

\newcommand\GCMI{\mathrm{GCMI}}
\newcommand\CMI{\mathrm{CMI}}
\newcommand\NC{\mathrm{NC}}

\usepackage{color}

\begin{document}

\fancyhead[LO]{Isotopies of conformal minimal surfaces}
\fancyhead[RE]{A.\ Alarc\'on and F.\ Forstneri\v c}
\fancyhead[RO,LE]{\thepage}

\thispagestyle{empty}

\vspace*{1cm}
\begin{center}
{\bf\LARGE Every conformal minimal surface in $\r^3$  
is isotopic to the real part  of a holomorphic null curve}   

\vspace*{0.5cm}

{\large\bf Antonio Alarc\'on $\;$ and $\;$ Franc Forstneri\v c}
\end{center}

\vspace*{1cm}

\begin{quote}
{\small
\noindent {\bf Abstract}\hspace*{0.1cm}
We show that for every conformal minimal immersion $u:M\to \r^3$ from an open 
Riemann surface $M$ to $\r^3$ there exists a smooth isotopy $u_t:M\to\r^3$ $(t\in [0,1])$ of 
conformal minimal immersions, with $u_0=u$, such that $u_1$ is the real part of a holomorphic null  curve $M\to \c^3$ 
(i.e. $u_1$ has vanishing flux). If furthermore $u$ is nonflat then $u_1$ can be chosen to have any prescribed flux and to be complete.

\vspace*{0.1cm}

\noindent{\bf Keywords}\hspace*{0.1cm} Riemann surfaces, minimal surfaces, holomorphic null  curves.

\vspace*{0.1cm}

\noindent{\bf MSC (2010):}\hspace*{0.1cm} 53C42; 32B15, 32H02, 53A10.}
\end{quote}


\section{The main results}
\label{sec:intro}

Let $M$ be a smooth oriented surface. A smooth immersion $u=(u_1,u_2,u_3):M\to \r^3$
is {\em minimal}  if its mean curvature vanishes at every point. 
The requirement that an immersion $u$ be {\em conformal} uniquely determines a 
complex structure on $M$. Finally, a conformal immersion is minimal
if and only if it is harmonic: $\Delta u=0$ (Osserman \cite{Osserman}).
A holomorphic immersion $F=(F_1,F_2,F_3):M\to\c^3$ of an open Riemann surface
to $\c^3$ is said to be a {\em null curve} if its differential 
$dF=(dF_1,dF_2,dF_3)$ satisfies the equation 
\[
		(dF_1)^2 + (dF_2)^2 + (dF_3)^2 =0.
\]
The real and the imaginary part of a null curve $M\to\c^3$ are conformal minimal immersions $M\to\r^3$.
Conversely, the restriction of a conformal minimal immersion $u:M\to \r^3$ 
to any simply connected domain $\Omega\subset M$ is the real part of a holomorphic 
null curve $\Omega\to \c^3$; $u$ is globally the real part of a null
curve if and only if its conjugate differential $d^c u$ satisfies
$\int_C d^c u=0$ for every closed curve $C$ in $M$.
This period vanishing condition means that $u$ admits a harmonic conjugate $v$, 
and $F=u+\imath v:M\to \c^3$ $(\imath=\sqrt{-1})$ is then a null curve.

In this paper we prove the following result
which further illuminates the connection between conformal minimal  
surfaces in $\r^3$ and holomorphic null curves in $\c^3$. 
We shall systematically use the term {\em isotopy} instead of the more standard 
{\em regular homotopy} when talking of smooth 1-parameter families of immersions. 

%
%
%
%
\begin{theorem}\label{th:main}
Let $M$ be an open Riemann surface. For every conformal minimal immersion $u:M\to \r^3$ 
there exists a smooth isotopy $u_t:M\to\r^3$ $(t\in [0,1])$ of conformal minimal immersions 
such that $u_0=u$ and $u_1=\Re F$  is the real part of a holomorphic null curve $F:M\to \c^3$. 
\end{theorem}

The analogous result holds for minimal surfaces in $\r^n$ for any $n\ge 3$, and the tools used in the proof
are available in that setting as well. On a compact bordered Riemann surface we also have an up to the boundary version 
of the same result (cf.\ Theorem \ref{th:main-bordered}).

Given a conformal minimal immersion $u\colon M\to\r^3$, the {\em flux map} of $u$ is the group homomorphism 
$\Flux_u\colon H_1(M;\z)\to\r^3$ on the first homology group of $M$ 
which is given on any closed curve $C$ in $M$ by
\begin{equation} \label{eq:fluxmap}
	\Flux_u(C)=\int_C d^c u.   
\end{equation} 
We can  view $\Flux_u$ as the element of the de Rham cohomology group $H^1(M;\r^3)$ determined by the closed real 
$1$-form $d^c u= \imath(\dibar u - \di u)$ with values in $\r^3$. (Note that $d^c u$ is closed precisely when $u$ is 
harmonic: $dd^cu=0$.) A conformal harmonic immersion $u\colon M\to\r^3$ is the real part of a holomorphic null curve 
$M\to\c^3$ if and only if the flux map $\Flux_u$ is identically zero. Hence Theorem \ref{th:main} can be expressed as follows.
\begin{center}
{\em Every conformal minimal immersion is isotopic to one with vanishing flux}. 
\end{center}

Every open Riemann surface $M$ is homotopy equivalent to a wedge of circles, and its first homology group $H_1(M;\z)$ 
is a free abelian group with at most countably many generators. If $M$ is of finite genus $g$ with $m$ topological ends 
then $H_1(M;\z)\cong \z^{2g+m-1}$.

We now decribe a more general existence result for isotopies of conformal minimal immersions.
Recall that an immersion $u:M\to\r^n$ is said to be {\em complete}  if the pullback $u^*(ds^2)$ of the Euclidean metric $ds^2$ 
on $\r^n$ is a complete metric on $M$. It is easily seen that a holomorphic null curve $M\to\c^3$ is complete if and only if  
its real part $M\to\r^3$ is complete (cf.\ Osserman \cite{Osserman}). An immersion $u:M\to\r^3$ is {\em nonflat} if its image 
$u(M)\subset\r^3$ is not  contained in any affine plane. It is easily seen that every flat conformal minimal immersion from a 
connected open Riemann surface is the real part of a holomorphic null curve (see Remark \ref{rem:flat}), so 
Theorem \ref{th:main} trivially holds in this case. For nonflat immersions we have the following second main result of this paper.

%
%
%
%
\begin{theorem}\label{th:main2}
Let $M$ be a connected open Riemann surface and let $\pgot\colon H_1(M;\z)\to\r^3$ be a group homomorphism.  
For every nonflat conformal minimal immersion $u:M\to\r^3$ there exists a smooth isotopy $u_t:M\to \r^3$ $(t\in [0,1])$ 
of conformal minimal immersions such that $u_0=u$, $u_1$ is complete, and $\Flux_{u_1}=\pgot$.  
Furthermore, if $u$ is complete then an isotopy as above can be chosen such that $u_t$ is complete for every $t\in[0,1]$.
In particular, every nonflat conformal minimal immersion $M\to\r^3$ is isotopic through conformal minimal immersions  
to the real part of a complete holomorphic null curve. 
\end{theorem}

It was shown by Alarc\'on, Fern\'andez, and L\'opez \cite{AFL1,AFL2} that every group homomorphism 
$\pgot\colon H_1(M;\z)\to\r^3$ is the flux map $\pgot=\Flux_u$ 
of a complete conformal minimal immersion $u\colon M\to\r^3$. The novel part of Theorem \ref{th:main2} is that one can deform 
an arbitrary nonflat conformal minimal immersion $M\to\r^3$ through a smooth family of such immersions to one 
that  is complete and has the given flux homomorphism.

Our results provide an initial step in the problem of homotopy classification
of conformal minimal immersions $M\to\r^3$ by their tangent maps; we discuss
this in Sect.\ \ref{sec:topology} below (see in particular Proposition \ref{prop:components}).
The subject of homotopy classification of immersions goes back to 
Smale \cite{Smale} and Hirsch \cite{Hirsch}, and it was later subsumed by Gromov's
h-principle in smooth geometry (see Gromov's monograph \cite{Gromov:book}).
The basic result of the Hirsch-Smale theory is that if $M$ and $N$ are smooth manifolds 
and $1\le \dim M<\dim N$ then regular homotopy classes of immersions $M\to N$ 
are in one-to-one correspondence with the homotopy classes of fiberwise injective vector bundle 
maps $TM\to TN$ of their tangent bundles; the same holds if $\dim M=\dim N$ 
and $M$ is not compact. This has been subsequently extended to several other classes
of immersions. In particular, Eliashberg and Gromov obtained the h-principle for 
holomorphic immersions of Stein manifolds to complex Euclidean spaces; see 
the discussion and references in \cite[Sec.\ 8.5]{F2011}.

Here we are interested in regular homotopy classes of conformal
minimal immersions of open Riemann surfaces into $\r^3$.  The $(1,0)$-derivative 
$\di u$ of a conformal (not necessarily harmonic) immersion $u:M\to\r^3$ gives a map 
$M\to \Agot^*=\Agot\setminus\{0\}$ with vanishing real periods 
into the punctured null quadric (\ref{eq:Agot}); {\em harmonic} (minimal) immersions 
correspond to holomorphic maps $M\to\Agot^*$. (See Sect.\ \ref{sec:prelim}.)
The question is whether regular homotopy classes 
of conformal minimal immersions are in one-to-one correspondence with the
homotopy classes of continuous maps $M\to \Agot^*$. 
(We wish to thank R.\ Kusner who pointed out (private communication)
the connection to the theory of spin structures on Riemann surfaces;
we refer to the preprint \cite{Kusner} by R.\ Kusner and N.\ Schmitt. 
Since $\Agot^*$ is an Oka manifold and every open Riemann surface $M$ is a Stein manifold, 
the homotopy classes of continuous maps $M\to \Agot^*$ 
coincide with the homotopy classes of holomorphic maps 
by the Oka principle \cite[Theorem 5.4.4]{F2011}.) One direction is provided
by \cite[Theorem 2.6]{AF2}:  {\em Every continuous map $M\to \Agot^*$ is
homotopic to the derivative of a holomorphic null immersion $M\to\c^3$, 
hence to the $(1,0)$-derivative of a conformal minimal immersion $M\to\r^3$.}
What remains unclear is whether two conformal minimal 
immersions $M\to\r^3$, whose $(1,0)$-derivatives are homotopic as maps $M\to \Agot^*$,
are regularly homotopic through conformal minimal immersions.
A more precise question is formulated as Problem \ref{pr:HE} below.

Another main result of the paper is an {\em h-Runge approximation theorem} 
for conformal minimal immersions of open Riemann surfaces to $\r^3$; see Theorem \ref{th:h-Runge}.
(Here, h stands for {\em homotopy}. This terminology is inspired by Gromov's h-Runge 
approximation theorem which plays a key role in the {\em Oka principle} 
for holomorphic maps from  Stein manifolds to elliptic and Oka manifolds;  cf.\ \cite{Gromov1989} 
and \cite[Chapter 6]{F2011}.) Basically our result says that a homotopy of 
conformal minimal immersions $u_t\colon U\to \r^3$ $(t\in [0,1])$, defined 
on a Runge open set $U$ in an open Riemann surface $M$ and such that $u_0$ 
extends to a nonflat conformal minimal immersion $M\to \r^3$, can be approximated uniformly 
on compacts in $U$ by a homotopy of conformal minimal immersions $\tilde u_t\colon M\to\r^3$
such that $\tilde u_0=u_0$. We also prove a version of this result with a fixed component function.
For the usual (non-parametric) version of this result  see \cite{AL1,AFL}. 

We now describe the content and the organization of the paper. 

In Sect.\ \ref{sec:prelim} we establish the notation and review the background.  In Sect.\ \ref{sec:lemma} 
we prepare the necessary results concerning the existence of loops with vanishing real or complex periods  
in the punctured null quadric $\Agot^*=\Agot\setminus\{0\}$ (\ref{eq:Agot}). After suitable approximation, 
using Mergelyan's theorem, a loop with vanishing real period represents the $(1,0)$-differential $\di u$ 
of a conformal minimal immersion $u\colon M\to \r^3$ along a closed embedded Jordan 
curve in our Riemann surface $M$; similarly, a loop with vanishing complex period 
represents the differential $dF=\di F$ of a holomorphic null curve $F\colon M\to\c^3$
along a curve in $M$. A reader familiar with Gromov's {\em convex integration theory}
\cite{Gromov:book,EM} will notice a certain similarity of ideas in the construction of
such loops. In order to use the Mergelyan approximation theorem and at the same time
keep the period vanishing condition we work with
{\em period dominating sprays of loops} (cf.\ Lemma \ref{lem:deform}),
using some results from our previous paper \cite{AF2} on holomorphic null curves.

In Sect.\ \ref{sec:proof1} we prove Theorem \ref{th:main} in the special case when $M$ has finite topology. 
The general case is treated in Sect.\ \ref{sec:proof}.

In Sect.\ \ref{sec:h-Runge} we prove Theorem \ref{th:h-Runge} (the h-Runge approximation theorem for 
conformal minimal immersions of open Riemann surfaces to $\r^3$).
By using this h-Runge theorem we obtain in Sect.\ \ref{sec:complete} several extensions of  
Theorem \ref{th:main}  to isotopies of {\em complete} conformal minimal immersions;
in particular, we prove Theorem \ref{th:main2}.

In Sec.\ \ref{sec:topology} we discuss the topology of the space of all conformal minimal 
immersions $M\to\r^3$ and we indicate several open questions related to the results in the paper.


\section{Notation and preliminaries} 
\label{sec:prelim}

A compact set $K$ in a complex manifold $X$ is said to be {\em holomorphically convex} (or
$\Oscr(X)$-convex) if for every point $p\in X\setminus K$ there exists a holomorphic function
$f\in\Oscr(X)$ satisfying $|f(p)|>\max_K |f|$. This notion is especially important if $X$ is a Stein manifold
(in particular, an open Riemann surface) in view of the Runge approximation theorem 
(also called the Oka-Weil theorem); see e.g.\ \cite{Hormander}.

Let $M$ be a Riemann surface. An immersion $u=(u_{1},u_{2},u_{3})\colon M\to \r^3$ is conformal
if and only if, in any local holomorphic coordinate $z=x+\imath y$ on $M$, 
the partial derivatives $u_x=(u_{1,x},u_{2,x},u_{3,x})$ and $u_y=(u_{1,y},u_{2,y},u_{3,y})$,
considered as vectors in $\r^3$, have the same Euclidean length and are orthogonal to each other 
at every point of $M$:
\begin{equation}\label{eq:conformal} 
	|u_x|=|u_y|>0, \qquad u_x\cdotp u_y=0.
\end{equation}
Equivalently,  $u_x \pm \imath u_y\in \c^3\setminus\{0\}$ are
{\em null vectors}, i.e., they lie in the {\em null quadric}
\begin{equation}
\label{eq:Agot}
	\Agot=\{z=(z_1,z_2,z_3)\in\c^3\colon z_1^2+z_2^2+z_3^2=0\}.
\end{equation}
We shall write $\Agot^*=\Agot\setminus\{0\}$. It is easily seen that $\Agot^*$
is a smooth closed hypersurface in $\c^3\setminus\{0\}$ which is the total space of a 
(nontrivial) holomorphic fiber bundle with fiber $\c^*$ over $\cp^1$
(see \cite[p.\ 741]{AF2}). In particular, $\Agot^*$ is an {\em Oka manifold} 
\cite[Proposition 4.5]{AF2}.

The exterior derivative on $M$ splits into the sum $d=\di+\dibar$ of the $(1,0)$-part $\di$
and the $(0,1)$-part $\dibar$. In any local holomorphic coordinate $z=x+\imath y$ on $M$ we have 
\begin{equation}\label{eq:di-u} 
	2\di u= (u_x - \imath u_y)dz, \quad 2\dibar u= (u_x + \imath u_y)d\bar z.
\end{equation}
Hence (\ref{eq:conformal}) shows that $u$ is conformal  if and only  if  
the differential $\di u=(\di u_1,\di u_2,\di u_3)$ satisfies the nullity condition
\begin{equation}\label{eq:sumuzero}
	(\di u_1)^2 + (\di u_2)^2 + (\di u_3)^2 =0.
\end{equation}

Assume now that $M$ is a connected open Riemann surface and that
$u\colon M\to \r^3$ is a conformal immersion. It is classical (cf.\ Osserman \cite{Osserman}) that 
\[\Delta u =2H  \nu,\] 
where $H\colon M\to\r$ denotes the mean curvature function of $u$,
$\nu\colon M\to \mathbb{S}^2\subset\mathbb{R}^3$ is the Gauss map of $u$, and $\Delta$ is the 
Laplacian operator with respect to the metric induced on $M$ by the Euclidean metric of $\r^3$ via $u$. 
Hence $u$ is minimal ($H=0$)  if and only if it is harmonic  ($\Delta u=0$). 
If $v$ is any local harmonic conjugate of $u$ then 
it follows from the Cauchy-Riemann equations that 
\[
	\di(u+\imath v) = 2\di u = 2\imath\, \di v.
\]
Thus $F=u+\imath v$ is a holomorphic immersion into $\c^3$ whose differential $dF=\di F=2\di u$
has values in $\Agot^*$ (\ref{eq:Agot}); i.e., a {\em null holomorphic immersion}. In particular, 
the differential $\di u$ of any conformal minimal immersion is a holomorphic 1-form
satisfying (\ref{eq:sumuzero}).

It is useful to introduce the {\em conjugate differential},  $d^c u= \imath(\dibar u - \di u)$. 
We have that
\begin{equation} \label{eq:dc}
	2\di u = du + \imath d^c u, \quad 
	dd^c u= 2\imath\,  \di\dibar u = \Delta u\cdotp  dx\wedge dy.
\end{equation}
If $u$ is harmonic (hence minimal) then $d^c u$ is a closed vector valued $1$-form on $M$, and
we have that $d^c u=dv$ for any local harmonic conjugate $v$ of $u$. The {\em flux map} 
of $u$  is the group homomorphism $\Flux_u\colon H_1(M;\z)\to\r^3$ given by
\[
	\Flux_u([C])=\int_C d^c u, \quad  [C]\in H_1(M;\z).
\]
The integral is independent of the choice of a path in a given homology class, and we 
shall write $\Flux_u(C)$ for $\Flux_u([C])$ in the sequel.
Furthermore, $u$ admits a global harmonic conjugate on $M$ if and only if 
the $1$-form $d^c u$ is exact on $M$, and this holds if and only if  
\begin{equation} \label{eq:moment1}
	\Flux_u(C)=\int_C d^c u =0 \quad \text{for\ every\ closed\ curve}\ C\subset M.
\end{equation}
We shall prove Theorem \ref{th:main} by finding an isotopy $u_t\colon M\to \r^3$ $(t\in [0,1])$ 
of conformal minimal immersions, with  $u_0=u$, such that $u_1$ satisfies the period condition (\ref{eq:moment1}). 

Fix a nowhere vanishing holomorphic $1$-form $\theta$ on $M$.  
(Such exists by the Oka-Grauert principle, cf.\ Theorem 5.3.1 in \cite[p.\ 190]{F2011}.)
It follows from (\ref{eq:di-u}) that $2\di u= f \theta$ where $f=(f_1,f_2,f_3)\colon M\to \Agot^*$
is a holomorphic map satisfying 
\begin{equation}\label{eq:WRperiods}
\int_C \Re (f\theta)= \int_C du=0\quad \text{for any closed curve $C$ in $M$}.
\end{equation}
Furthermore, we have $u=\Re F$ for some null holomorphic immersion $F\colon M\to\c^3$ if and only if 
$\int_C f\theta = 0$ for all closed curves $C$ in $M$. 
The meromorphic function on $M$ given by
\begin{equation}\label{eq:gaussmap}
g :=\frac{f_3}{f_1-\imath f_2}
\end{equation}
is the stereographic projection of the Gauss map of $u$, and the map $f=2\di u/\theta$ 
can be recovered from the pair $(g,f_3)$ by the expression
\begin{equation}\label{eq:WR}
f=\left(\frac12\Big(\frac1{g}-g\Big) , \frac{\imath}2\Big(\frac1{g}+g\Big) , 1 \right)f_3.
\end{equation}
The pair $(g,f_3\theta)$ is called the {\em Weierstrass data} of $u$. 
The Riemannian metric $ds_u^2$ induced on $M$ by the Euclidean metric of $\r^3$ via the immersion $u$ equals
\begin{equation}\label{eq:metric}
ds_u^2= \frac12 \, |f\theta|^2=\frac14 \left( \frac1{|g|}+|g|\right)^2|f_3|^2|\theta|^2.
\end{equation}
We denote by $\dist_u$ the distance function induced on $M$ by $ds_u^2$.
Conversely, given a meromorphic function $g$ and a holomorphic function $f_3$ on $M$ such that the map $f$ \eqref{eq:WR} 
has no poles, then $f$ assumes values in $\Agot$. If in addition $f$ does not vanish anywhere on $M$ and satisfies 
\eqref{eq:WRperiods} then $ds_u^2>0$ everywhere on $M$ (see \eqref{eq:metric}), and $f\theta$ integrates to a 
conformal minimal immersion $u\colon M\to\r^3$ given  by $u(x)=\int^x \Re(f\theta)$ for $x\in M$.

%
%
%
%
\begin{remark}\label{rem:flat}
If $u$ is a {\em flat} (planar) immersion, in the sense that the image
$u(M) \subset\r^3$ lies in an affine $2$-plane in $\r^3$, 
we may assume after an orthogonal change of coordinates that $u_3=const$.
In this case (\ref{eq:di-u}) implies  $\di u_1=\pm \imath \di u_2$ which gives 
$d^c u_1=\pm du_2$ and $d^c u_2=\pm du_1$, so $d^c u$ is exact. 
This shows that {\em every flat conformal minimal immersion is the real part of a holomorphic 
null curve $F\colon M\to\c^3$.}  In this case the image of $dF$ lies in a ray $\c \nu\subset \c^3$ spanned  
by a null vector $0\ne \nu \in \Agot$, and $F(M)$ is contained in an affine complex line
$a+\c \nu \subset \c^3$.   Such {\em flat null curves} are precisely those that are {\em degenerate} 
in the sense of \cite[Definition 2.2]{AF2}. Note that every open Riemann surface $M$ admits flat null 
holomorphic immersions $M\to\c^3$ of the form  $F(x)=e^{g(x)} \nu$ $(x\in M)$,
where $0\ne \nu\in \Agot$ is a null vector and $g\in \Oscr(M)$ 
is a holomorphic function  without critical points \cite{GN}.
\end{remark}

We denote by $\CMI(M)$ the set of all conformal minimal immersions $M\to \r^3$
and by $\CMI_*(M)\subset \mathrm{CMI}(M)$ the subset consisting of
all nonflat immersions. By $\NC(M)$ we denote the space of all null holomorphic immersions 
$F\colon M\to\c^3$, and $\NC_*(M)$  is the subset of $\NC(M)$ consisting of nonflat immersions.
These spaces are endowed with the compact-open topology. We have natural inclusions
\[
	\Re\NC(M) \hookrightarrow \CMI(M), \quad  \Re\NC_*(M) \hookrightarrow \CMI_*(M),
\]
where $\Re\NC(M)=\{\Re F\colon F\in \NC(M)\}$ is the kernel of the flux map (\ref{eq:moment1}) on $\CMI(M)$.

If $K\subset M$ is a compact subset, we denote by $\CMI(K)$ the set of all conformal minimal immersions of unspecified open neighborhoods of $K$ into $\r^3$, and by $\CMI_*(K)\subset \mathrm{CMI}(K)$ the subset consisting of all immersions which are not flat on any connected component. We define $\NC(K)$ and $\NC_*(K)$ in the analogous way.

%
%
%
%
%

\section{Loops with prescribed periods in the null quadric}\label{sec:lemma}

Assume that $C$ is a closed, embedded, oriented, real analytic curve in $M$. 
There are an open set $W\subset M$ containing $C$ and a biholomorphic map 
$z\colon W \to \Omega$ onto an annulus $\Omega=\{z\in \c\colon r^{-1} <|z|< r\}$
taking $C$ onto the positively oriented unit circle 
\[
	\mathbb{S}^1=\{z\in \c: |z|=1\}.
\]  
The exponential map $\c\ni \zeta= x+\imath y \mapsto \exp(2\pi \imath\, \zeta)\in \c^*$ 
provides a universal covering of the annulus $\Omega$ by the strip 
$\Sigma =\{x+\imath y\colon x\in \r,\ |y|< (2\pi)^{-1}\log r \} \subset \c$, 
mapping the real axis $\r=\{y=0\}$ onto the circle $\mathbb{S}^1\cong \r/\z$. 
We shall view $\zeta=x+\imath y$ as a {\em uniformizing coordinate} on $W$, with $C=\{y=0\}$. 
The restriction of a conformal harmonic immersion $u\colon M\to\r^3$ to $W$ 
is given in this coordinate by a $1$-periodic conformal harmonic immersion 
$U\colon\Sigma\to \r^3$. Along $y=0$ we have 
\begin{equation}\label{eq:Taylor}
	U(x+\imath y) = h(x) - g(x) y + O(y^2)
\end{equation}
where $h(x)=U(x+\imath 0)$ and $g(x)=-U_y(x+\imath 0)$ are smooth $1$-periodic maps $\r\to\r^3$, 
$h$ is an immersion, and the remainder $O(y^2)$ is bounded by $cy^2$ for some $c>0$ 
independent of $x$. We have that
\[
	2\frac{\di}{\di \zeta} U(x+\imath y)|_{y=0} = \left(U_x - \imath U_y\right)|_{y=0} = 
	h'(x)+\imath  g(x).
\]
The conformality condition (\ref{eq:conformal}) implies that 
\begin{equation}\label{eq:conformal2}
	g(x) \cdotp  h'(x)=0 \quad \text{and} \quad |g(x)| = |h'(x)|>0
	\quad \text{hold\ for\ all} \ x\in \r.
\end{equation}
We also have that $d^c U =-U_y dx+U_x dy$ and hence 
\[
	\int_C d^c u= \int_0^1  d^c U= - \int_0^1 U_y(x+\imath 0) dx = \int_0^1 g(x) dx.
\] 
Condition (\ref{eq:conformal2}) is equivalent to saying that the map
$\sigma = h'+\imath g\colon \r \to \Agot^* = \Agot\setminus\{0\}$  
is a loop in  the null quadric (\ref{eq:Agot}) whose real part has vanishing period: 
\[
	\int_0^1 \Re \sigma(x) dx=\int_0^1 h'(x)dx =0.
\] 
In this section we prove that for every such $\sigma$ 
and for any vector $v\in\r^3$ there is an isotopy of $1$-periodic maps 
$\sigma_t\colon \r \to \Agot^*$ $(t\in [0,1])$ such that $\sigma_0=\sigma$ and
\begin{equation}\label{eq:isotopy}
	\int_0^1 \Re \sigma_t(x) dx=0\ \ \text{for\ all}\ t\in [0,1], \qquad
	\int_0^1 \Im \sigma_1(x) dx=v.	
\end{equation}
(See Lemmas \ref{lem:zero-period} and \ref{lem:periods}.) 
This will be one of the main steps in the proof of Theorems \ref{th:main} and \ref{th:main2}.
(The special case $v=0$  will be of importance for the proof of Theorem  \ref{th:main}.)

We denote by $\Igot$ the space of all smooth $1$-periodic immersions $\r\to\r^3$
endowed with the $\Cscr^\infty$ topology. We identify an immersion $h\in \Igot$ 
with a smooth immersion $\mathbb{S}^1=\r/\z\to\r^3$ of the circle into $\r^3$. 
We shall say that $h\in \Igot$ is {\em nonflat} on a segment $L\subset [0,1]$ 
if the image $h(L)$ is not contained in any affine plane of $\r^3$.

%
%
%
%
\begin{lemma}\label{lem:immersions}
Any pair of immersions $h_0,h_1\in \Igot$ can be connected by a smooth path of immersions $h_t\in\Igot$ $(t\in [0,1])$. 
If furthermore $h_0,h_1\in \Igot$ agree on a proper closed subinterval $I\subset [0,1]$ then the isotopy $h_t$ can be 
chosen fixed on $I$. If $h_0|_L$ is nonflat on a segment $L\subset [0,1]$ then we can arrange that $h_t|_L$ is 
nonflat for every $t\in [0,1]$.
\end{lemma}

\begin{proof}
Connect $h_0$ and $h_1$ by $\tilde h_t=(1-t)h_0+h_1$ $(t\in [0,1])$. The jet transversality
theorem shows that a generic perturbation $\{h_t\}$ of $\{\tilde h_t\}$ 
with fixed ends at $t\in\{0,1\}$ (or one that is in addition fixed on a subinterval $I\subset [0,1]$
on which $h_0$ ad $h_1$ agree) yields a smooth isotopy of immersions connecting $h_0$ and $h_1$.
The nonflatness condition can be satisfied by a generic deformation; observe also that being flat is a closed
condition.
\end{proof}

\begin{lemma}\label{lem:zero-period}
Every smooth $1$-periodic immersion $h_0\colon\r\to\r^3$  (i.e., $h_0\in \Igot$) can be 
approximated in $\Igot$ by a smooth $1$-periodic immersion $h\colon\r\to\r^3$ for which there 
exists a smooth  $1$-periodic map $g\colon\r\to\r^3$ satisfying the following properties:
\begin{enumerate}[\rm (i)]
\item   $g(x) \cdotp h'(x)=0$ and $|g(x)| = |h'(x)|>0$ for all $x\in \r$, and 
\item  $\int_0^1 g(x)dx = 0$.
\end{enumerate}
A map $g$ with these properties can be chosen in any given homotopy class of sections 
of the circle bundle over $\r/\z\cong \mathbb{S}^1$ determined by the condition {\rm (i)}.
\end{lemma}

%
%
%
%
\begin{remark}\label{rem:bundleE}
The last sentence in Lemma \ref{lem:zero-period} requires a comment. 
Denote the coordinates on $\c^3$ by $z=\xi + \imath \eta$, with $\xi,\eta\in\r^3$.
Let $\pi\colon \c^3=\r^3\oplus \imath \r^3 \to\r^3$ be the projection
$\pi(\xi +\imath \eta)=\xi$. Then $\pi^{-1}(0)\cap\Agot =\{0\}$ and
$\pi\colon \Agot^*\to \r^3\setminus \{0\}$ is a real analytic fiber bundle with circular fibers given by 
\begin{equation}\label{eq:fiberbundle}
	 \Agot \cap \pi^{-1}(\xi) = \{\xi+\imath \eta\in\c^3: \xi\cdotp \eta=0,\ |\xi|=|\eta|\}
	 \cong \mathbb{S}^1, \quad \xi \in \r^3\setminus \{0\}.
\end{equation}
An immersion $h\in \Igot$ determines the circle bundle 
$E_h=(h')^*(\Agot^*)\mapsto \r/\z\cong \mathbb{S}^1$ 
(the pull-back of $\Agot^*\to \r^3\setminus\{0\}$ by the derivative $h'\colon\r\to\r^3\setminus\{0\}$), and 
sections of $E_h$ are $1$-periodic maps $g\colon\r\to \r^3$ satisfying condition  (i) in Lemma \ref{lem:zero-period}.  
Every oriented circle bundle $E_h \to \mathbb{S}^1$ is trivial, and the set of homotopy classes
of its sections $\mathbb{S}^1\to E_h$ can be identified with the fundamental group $\pi_1(\mathbb{S}^1)=\z$. 
\end{remark}

%
%
\begin{proof}[Proof of Lemma \ref{lem:zero-period}]
Pick  a $\delta >0$ with $3\delta<1$. We  approximate $h_0$ by an
immersion $\tilde h_0\in \Igot$ whose derivative is constant on 
$J=[0,3\delta]$ and such that $\tilde h_0$ agrees with $h_0$ outside a slightly bigger interval $J'\supset J$. 
(We think of intervals in $[0,1]$ as arcs in the circle $\r/\z=\mathbb{S}^1$. The rate of approximation
of $h_0$ by $\tilde h_0$  will of course depend on $\delta$ which we are free to choose as small as we wish.)
Replacing $h_0$ by $\tilde h_0$ we assume from now on that $h_0$ satisfies these properties. 

Let $\be_1,\be_2,\be_3$ denote the standard basis of $\r^3$.
After an orthogonal rotation and a dilation on $\r^3$  
we may assume that $h'_0(x)=\be_1$ for $x\in [0,3\delta]$. 
Let $\mathbb{B}$ denote the closed  unit ball in $\r^3$. Pick a number $0<\epsilon<1$ and
a family of immersions $h_p\in \Igot$, depending smoothly on 
$p=(p_1,p_2,p_3)= p_1\be_1+p_2\be_2+p_3\be_3\in \mathbb{B}$ 
and with $h_0$ the given immersion, such that
\[
	h'_p(x) = \begin{cases} \be_1 +\epsilon p, & \text{if}\ x\in [0,\delta], \\
                                           \be_1,                     & \text{if}\   x\in [2\delta,3\delta].
                      \end{cases}                
\]
We choose $h_p$ to agree with $h_0$ outside a small neighborhood of the indicated intervals,
so every $h_p$ is close to $h_0$ (depending on the choice of $\epsilon$).
Let $A\in O(3)$ be the orthogonal linear transformation on $\r^3$ given by
\[
	A\be_1=\be_2, \quad A\be_2=-\be_1,  \quad A\be_3=\be_3.
\]
Define a locally constant map $\tilde g_p\colon [0,\delta]\cup  [2\delta,3\delta] \to \r^3$ by setting 
\[
	\tilde g_p(x) = \begin{cases} \be_2+\epsilon Ap - \epsilon^2 \frac{p_3^2}{1+\epsilon p_1} \be_1, 
									& \text{if}\ x\in [0,\delta]; \\
                                           -\be_2,                     & \text{if}\   x\in [2\delta,3\delta].
                      \end{cases}                
\]
Since $0<\epsilon<1$ and $p\in \mathbb{B}$, we have $1+\epsilon p_1>0$ and hence $\tilde g_p$ is well defined;
furthermore, we have that $\tilde g_p(x) \ne 0$ for all $x\in [0,\delta]\cup[2\delta,3\delta]$ and $p\in \mathbb{B}$. 
A calculation shows that $\tilde g_p(x)\cdotp h'_p(x)=0$ for 
$x \in [0,\delta] \cup [2\delta,3\delta]$  and
\begin{equation}\label{eq:integral}
		\int_0^\delta \tilde g_p(x)dx + \int_{2\delta}^{3\delta} \tilde g_p(x)dx = 
		\epsilon\delta \left(Ap - \epsilon  \frac{p_3^2}{1+\epsilon p_1} \be_1 \right) .
\end{equation}
Set
\begin{equation}\label{eq:g-p}
     g_p(x) = \begin{cases} 
                      \frac{|\be_1+\epsilon p|}{|\tilde g_p(x)|} \, \tilde g_p(x), &   \text{if}\ x\in [0,\delta]; \\
                        \tilde g_p(x)=-\be_2,  & \text{if}\   x\in [2\delta,3\delta].
                    \end{cases}                
\end{equation}
Clearly $|g_p(x)|=|h'_p(x)|$ and $g_p(x)\cdotp h'_p(x)=0$ for all $x \in [0,\delta] \cup [2\delta,3\delta]$
and $p\in \mathbb{B}$. For $x \in [0,\delta]$ we also have  
$\tilde g_p(x) = \be_2+\epsilon Ap +O(\epsilon^2)= A(\be_1+\epsilon p) + O(\epsilon^2)$.
Since $A$ is orthogonal, we get
$|\tilde g_p(x)| = |A(\be_1+\epsilon p)| + O(\epsilon^2) = |\be_1+\epsilon p|+O(\epsilon^2)$
and hence
\[
	\frac{|\be_1+\epsilon p|}{|\tilde g_p(x)|} = 
	\frac{|\be_1+\epsilon p|}{|\be_1+\epsilon p| + O(\epsilon^2)}
        =	1+ O(\epsilon^2).
\]
This shows that the rescaling in the definition of $g_p$ (\ref{eq:g-p}) 
changes the integrals in (\ref{eq:integral}) 
only  by a term of size $O(\epsilon^2\delta)$, so we have 
\begin{equation}\label{eq:integral2}
		\int_0^\delta g_p(x)dx + \int_{2\delta}^{3\delta} g_p(x)dx = 
		\epsilon\delta \left(Ap  + O(\epsilon) \right) .
\end{equation}

We now extend each $g_p$ to a smooth $1$-periodic map $\r\to\r^3$, depending  smoothly on 
$p\in \mathbb{B}$, such that the following conditions hold:
\begin{itemize}
\item[\rm (a)]    $g_p(x) \cdotp  h'_p(x)=0$ and $|g_p(x)|=|h'_p(x)|>0$ for all $x\in \r$ and $p\in \mathbb{B}$, and
\item[\rm (b)]    $\big| \int_\delta^{2\delta} g_p(x) dx +\int_{3\delta}^{1} g_p(x)dx \big| < \frac{1}{3} \epsilon\delta$
for all $|p|=1$.
\end{itemize}
Condition (a) is compatible with the definition of $g_p$ on $[0,\delta] \cup [2\delta,3\delta]$. 
Condition (b) can be achieved by choosing $g_p(x)$ to spin sufficiently fast 
(with nearly constant angular velocity)
along the circle bundles defined by Condition (a) as $x$ traces the intervals in the two integrals. 
Since spinning in both directions is allowed, and we can change the direction
on short intervals with an arbitrarily small contribution to the integral, 
we can arrange that $g_p$ belongs to any given homotopy class 
of sections (independent of $p\in \mathbb{B}$).  

Choosing $\epsilon>0$ small enough, Condition (b) together with (\ref{eq:integral2})  implies
\[
	\left| \int_0^1 g_p(x)dx - \epsilon\delta Ap \, \right| < \frac{\epsilon\delta}{2},
	\qquad  p\in \mathbb{B}.
\]
Since $|Ap|=|p|$ for all $p\in\r^3$, it follows that the map 
$\mathbb{B}\ni p \longmapsto \int_0^1 g_p(x)dx \in \r^3$ 
is nowhere vanishing on the sphere $\mathbb{S}^2=b\mathbb{B}=\{|p|=1\}$ 
and the restricted map $\mathbb{S}^2\mapsto \r^3\setminus \{0\}\simeq \mathbb{S}^2$ has 
the same degree as the map $p\mapsto Ap$, which is one.
Hence there exists a point $p\in \mathring{\mathbb B}$ for which $\int_0^1 g_p(x)dx =0$. 
The pair $(g,h)=(g_p,h_p)$ then clearly satisfies Lemma \ref{lem:zero-period}.
\end{proof}

We shall also need the following version of Lemma \ref{lem:zero-period}
in which all maps remain fixed on a proper subinterval of $[0,1]$ and the period
of $g_1$ equals any given vector in $\r^3$.

%
%
%
%
\begin{lemma}\label{lem:periods}
Let $g_0,h_0\colon \r\to\r^3$ be a smooth $1$-periodic maps satisfying  Condition (\ref{eq:conformal2}),
and let $v\in\r^3$. Given a proper closed subinterval $I\subset [0,1]$, there exist smooth isotopies 
of $1$-periodic maps $g_t,h_t\colon\r\to\r^3$ $(t\in [0,1])$ such that 
\begin{itemize}
\item $g_t(x)=g_0(x)$ and $h_t(x)=h_0(x)$ for $x\in I$ and $t\in [0,1]$, 
\item Condition (\ref{eq:conformal2}) holds for $(g_t,h_t)$ for every $t\in [0,1]$, and 
\item $\int_0^1 g_1(x)dx = v$.
\end{itemize}
Furthermore, given a segment $I'\subset [0,1]$ such that $h_0|_{I'}$ is nonflat, we can choose $h_t,g_t$ 
as above such that  $h_t|_{I'}$ is nonflat for every $t\in [0,1]$.
\end{lemma}

\begin{proof}
Choose a pair of nontrivial closed intervals $J,L\subset [0,1]$ such that
the intervals $I,J,L\subset [0,1]$ are pairwise disjoint. We may assume 
that $J=[0,3\delta]$ for a small $\delta>0$. (As before, we consider these intervals
as arcs in the circle $\r/\z=\mathbb{S}^1$.) We explain the individual moves without 
changing the notation at every step.

We begin by deforming the pair $(g_0,h_0)$,  keeping it fixed on the segment $I$, 
such that for $x\in J$ we have  $h'_0(x)=\be_1$ and $g_0(x)=\be_2$. 
Consider the $1$-periodic map 
\[
	\sigma_0:= h'_0+\imath g_0 \colon \r\to \Agot^*=\Agot\setminus \{0\}.
\] 
Set 
\[
	w=\int_{[0,1]\setminus L} \sigma_0(x)dx \in \c^3.
\]
By \cite[Lemma 7.3]{AF2}  there is a smooth $1$-periodic map 
$\sigma_1\colon \r\to \Agot^*$ which agrees with $\sigma_0$ on $[0,1]\setminus L$ 
such that $\int_L \sigma_1(x)dx \in \c^3$ is arbitrarily close to $\imath v-w \in\c^3$, and 
hence $\int_0^1\sigma_1(x)dx$ is close to $\imath v$. 
(The main point of the proof is that $\Agot$ is connected and its convex hull  equals $\c^3$.)
By general position we may assume that $\Re\sigma_1$ does not assume 
the value $0\in \r^3$. By another small correction of $\sigma_1$ on $L$ we may also arrange that 
$\int_0^1 \Re \sigma_1(x)dx =0$, while $\sigma_1$ still assumes values
in $\Agot^*$ and $\int_0^1 \Im \sigma_1(x)dx$
remains close to $v$. Fix a point $x_0\in I$ and set 
\[
	h_1(x)=h_0(x_0)+ \int_{x_0}^x \Re \sigma_1(s) ds, \quad x\in\r.
\]
Then $h_1\in \Igot$ is a smooth $1$-periodic immersion which agrees with $h_0$
on $[0,1]\setminus L$ and satisfies $h'_1=\Re\sigma_1$.
By Lemma \ref{lem:immersions} we can connect $h_0$ to $h_1$ by a smooth isotopy 
of immersions $h_t\in \Igot$ such that $h_t(x)=h_0(x)$ for every  $t\in [0,1]$ and $x\in [0,1]\setminus L$. 

By the argument given in Remark \ref{rem:bundleE} we can cover the homotopy 
$h'_t\colon \r\to\r^3\setminus\{0\}$ by a smooth homotopy of $1$-periodic maps 
$g_t\colon\r\to \r^3$ such that $g_0$ is the given initial map,
the homotopy is fixed on $[0,1]\setminus L$, and $(g_t,h_t)$ satisfies 
Condition  (\ref{eq:conformal2}) for every $t\in [0,1]$. The maps $g_1$ and $\Im\sigma_1$
are sections of the circle bundle $E_1=(h'_1)^*E\to \mathbb{S}^1$, but they need not be homotopic.
To correct this, we replace $\Im\sigma_1$ by another section of the same bundle 
whose period is still close to $v$ and which is homotopic to $g_1$
(see the last sentence in Lemma \ref{lem:zero-period}). 
It is then possible to choose the homotopy $g_t$ as above connecting $g_0$ to $g_1=\Im \sigma_1$. 
The homotopy of smooth $1$-periodic maps $\sigma_t=h'_t+\imath g_t\colon \r\to \Agot^*$
($t\in [0,1]$) satisfies all the required properties, except that the period $\int_0^1 g_1(x)dx \in\r^3$
is only close to $v \in \r^3$. It remains to make this period  exactly equal to $v$
by another small deformation of the pair $(g_1,h_1)$ that is fixed on the segment $I$. 
This can be achieved  by the perturbation device on the segment 
$J=[0,3\delta]$ described in the proof of Lemma \ref{lem:zero-period}.
\end{proof}

We shall also need a period perturbation lemma for (finite unions of)
loops in the null quadric $\Agot$ (see \ref{eq:Agot})). Let $\Lcal=\Cscr^\infty(\mathbb{S}^1,\Agot^*)$ denote 
the space of smooth loops in $\Agot^*=\Agot\setminus\{0\}$.
We may think of $\sigma\in \Lcal$ as a smooth $1$-periodic map $\r\to\Agot^*$;
in particular, we write 
\[
	\int_{\mathbb{S}^1}\sigma=\int_0^1\sigma(x)dx \in\c^3.
\]
We identify the tangent space $T_z\Agot\subset T_z\c^3$ at a point 
$z\in \Agot^*$ with a complex $2$-dimensional subspace of $\c^3$. 

%
%
\begin{definition}\label{def:nondegenerate}
A loop $\sigma\in \Lcal$ is said to be {\em nondegenerate} on a segment $I\subset \mathbb{S}^1$
if the family of tangent spaces $\{T_{\sigma(x)}\Agot : x\in I\}$ spans $\c^3$.
\end{definition}

Since $\Agot$ is a complex cone, the tangent space $T_z \Agot$ at any point 
$0\ne z\in \Agot$  is spanned by $z$ together with one more vector, 
and $T_z \Agot =T_w \Agot$ for any point $w=\lambda z$ with $\lambda\ne 0$.
It follows that a loop $\sigma\in \Lcal$ is nondegenerate on the segment $I\subset \mathbb{S}^1$
if and only if the image $\sigma(I)$ is not contained in any ray $\c\cdotp \nu \subset\Agot$ of the 
null quadric. 

A continuous map $\sigma\colon W\to \Lcal$ from a complex manifold $W$ to the loop space $\Lcal$ is 
naturally identified with a continuous map $\sigma\colon W \times \mathbb{S}^1\to \c^3$. 
Such a map is said to be holomorphic if for every $x\in \mathbb{S}^1$ the map 
$\sigma(\cdotp,x)\colon W\to \c^3$ is holomorphic; in such case we shall also say
that the family $\sigma_w=\sigma(w,\cdotp)\in \Lcal$ is holomorphic in $w\in W$.

Consider the period map $\Pcal\colon\Lcal\to \c^3$ defined by
\[
	\Pcal(\sigma) = \int_0^1 \sigma(x)\, dx \in \c^3,\quad \sigma\in \Lcal.
\]

%
%
%
%
\begin{lemma}[Period dominating sprays of loops] \label{lem:deform}
Let $I\subset \mathbb{S}^1$ be a nontrivial segment.
If a loop $\sigma\in \Lcal=\Cscr^\infty(\mathbb{S}^1,\Agot^*)$ is nondegenerate on $I$ (Def.\ \ref{def:nondegenerate}) then
there is a holomorphic family of loops $\{\sigma_w\}_{w\in W} \in \Lcal$, 
where $W \subset \c^3$ is a ball centered at $0\in \c^3$, such that $\sigma_0=\sigma$,
$\sigma_w(x)=\sigma(x)$ for all $x\in \mathbb{S}^1\setminus I$ and $w\in W$, and 
\[
	\frac{\partial \Pcal(\sigma_w)}{\partial w} \bigg|_{w=0} :T_0\c^3\cong \c^3\longrightarrow \c^3
	\quad \text{is an isomorphism}.
\]	
More generally, given a family of loops $\sigma_q\in \Lcal$
depending continuously on a parameter $q$ in a compact Hausdorff space $Q$ 
such that $\sigma_q$ is nondegenerate on the segment $I$ for every $q\in Q$, 
there is a ball $0\in W\subset \c^N$ for some $N\in\n$ and a 
continuous family of loops $\sigma_{q,w}\in \Lcal$ $(q\in Q,\ w\in W)$, 
depending holomorphically on $w\in W$, such that 
$\sigma_{q,0}=\sigma_q$ for all $q\in Q$, $\sigma_{q,w}=\sigma_q$ on $\mathbb{S}^1\setminus I$ 
for all $q\in Q$ and $w\in W$, and 
\begin{equation}\label{eq:subm}
	\frac{\partial \Pcal(\sigma_{q,w})}{\partial w} \bigg|_{w=0}\colon T_0\c^N \cong \c^N\longrightarrow  \c^3
	\quad \text{is surjective for every}\ q\in Q.
\end{equation}
\end{lemma}

\begin{definition}\label{def:spray}
A family  of loops $\{\sigma_w\}_{w\in W}$ satisfying the domination condition (\ref{eq:subm}) is 
called a  holomorphic period dominating  spray of loops.  
\end{definition}

For a general notion of a (local) dominating holomorphic spray see for instance \cite[Definition 4.1]{DF2007} or 
\cite[Def. 5.9.1]{F2011}. Period dominating sprays  were first constructed in \cite[Lemma 5.1]{AF2}.

\begin{proof}
The main idea is contained in the proof  of \cite[Lemma 5.1]{AF2}; 
we outline the main idea for the sake of readability. Since $\sigma$ is nondegenerate on  $I$, 
there are points $x_1,x_2,x_3\in \mathring I$ and holomorphic vector 
fields $V_1,V_2,V_3$ on $\c^3$ which are tangential to the quadric $\Agot$ (\ref{eq:Agot})
such that the vectors $V_j(\sigma(x_j))\in\c^3$ for $j=1,2,3$ 
are a complex basis of $\c^3$. Choose 
a smooth function $h_j\colon \mathbb{S}^1\to \c$ supported on a short segment $I_j\subset I$ 
around the point $x_j$ for $j=1,2,3$.  Let $\phi^j_t$ denote the flow of the vector field $V_j$ 
for time $t\in \c$ near $0$. It is easily seen that for suitable choices
of the functions $h_j$ the holomorphic spray of loops $\sigma_w\in \Lcal$, given
for any $w=(w_1,w_2,w_3)\in \c^3$ sufficiently close to the origin by
\[
		\sigma_w(x) = \phi^1_{w_1 h_1(x)}\circ \phi^2_{w_2 h_2(x)} \circ \phi^3_{w_3 h_3(x)} \sigma(x) \in \c^3,
		 \quad x\in \mathbb{S}^1,
\]
enjoys the stated properties; in particular, it is period dominating.
The same proof applies to a continuous family  of loops $\{\sigma_q\}_{q\in Q}\subset\Lcal$
by using compositions of flows of finitely many holomorphic vector fields tangential to $\Agot$. 
\end{proof}

%
%
%
%
\begin{remark}\label{rem:deform}
Lemma \ref{lem:deform} also applies to a finite union 
$C=\bigcup_{j=1}^l C_j$ where each $C_j\cong \mathbb{S}^1$ 
is an embedded oriented analytic Jordan curve in a Riemann surface. Let $\sigma\colon C\to \Agot^*$
be a smooth map. Given pairwise disjoint segments 
$I_j\subset C_j\setminus \bigcup_{i\ne j} C_i$ $(j=1,\ldots,l)$ 
such that $\sigma$ is nondegenerate on every $I_j$ in the sense of 
Def.\ \ref{def:nondegenerate}, the same proof furnishes a 
holomorphic family $\sigma_w\in \Cscr^\infty(C,\Agot^*)$
for $w$ in a ball $0\in W\subset  \c^{3l}$ such that 
$\sigma_0=\sigma$, $\sigma_w$ agrees with $\sigma$ on  $C\setminus \bigcup_{j=1}^l I_j$
for every  $w\in W$, and the differential at $w=0$ of the map 
$P=(P_1,\ldots,P_l) \colon W\to (\c^3)^l$ with the components 
\[
	P_j(w)=\int_{C_j} \sigma_w\in \c^3
\]
is an isomorphism. The analogous result also holds for 
homotopies of maps $\sigma_q\in \Cscr^\infty(C,\Agot^*)$
with the parameter $q\in Q$ in a compact Hausdorff space.
In the present paper we shall use it for 1-parameter homotopies, with the parameter $q=t\in [0,1]$.
\qed \end{remark}

%
%
%
%
\begin{remark}\label{rem:asympt-holo}
Results in this section apply to real analytic Jordan curves in Riemann surfaces via 
a holomorphic change of coordinates along the curve, mapping the curve onto the circle $\mathbb{S}^1\subset\c$. 
However, one can also use them for smooth curves by applying an asymptotically holomorphic  change of coordinates 
(i.e., one whose $\dibar$-derivative vanishes along the curve),
mapping the curve onto the  circle $\mathbb{S}^1$. Such changes of coordinates,
which always exist along  a smooth curve, respect the conditions (\ref{eq:Taylor}) and
\eqref{eq:conformal2} which concern first order jets. 
However, real analytic curves suffice for the applications in this paper.
\qed \end{remark}


\section{Proof of Theorem \ref{th:main} for surfaces with finite topology} \label{sec:proof1}
In this section we prove Theorem \ref{th:main} for Riemann surfaces $M$ with finitely generated first 
homology group $H_1(M;\z)\cong\z^l$, $l\in\n$. 

Without loss of generality we may assume that $M$ is connected and 
the immersion $u$ in Theorem \ref{th:main} is nonflat, $u \in \CMI_*(M)$
(cf.\ Remark \ref{rem:flat}).  Set $u_0:=u$.

Fix $p_*\in M$. There exist embedded, closed, oriented, real analytic curves
$C_1,\ldots,C_l$ in $M$ such that $C_i\cap C_j=\{p_*\}$ when $i\ne j$, the 
homology classes $[C_j]\in H_1(M;\z)$ are a basis of the first homology group $H_1(M;\z)$, 
$M$ retracts onto $C=\bigcup_{j=1}^{l} C_j$, and $C$ is $\Oscr(M)$-convex.
Along each curve $C_j$ we choose a uniformizing holomorphic coordinate $\zeta_j =x+\imath y$
(see Sect.\ \ref{sec:lemma}). In this coordinate we can represent the differential $2\di u_0$ 
on $C_j$ by a pair of smooth $1$-periodic maps $g_{j,0},h_{j,0}\colon\r \to\r^3$ such that 
\[
	\sigma_{j,0} :=h'_{j,0} +\imath g_{j,0}\colon\r\to\Agot^*
	\quad\text{and}\quad   2\di u_0 = \sigma_{j,0}d\zeta_j\ \ \text{on}\ C_j.
\]
Lemma \ref{lem:periods} furnishes smooth homotopies of $1$-periodic maps $h_{j,t},g_{j,t}\colon\r\to\r^3$
that are fixed near the intersection point $C_i\cap C_j=\{p_*\}$ such that
\begin{equation}\label{eq:sigma}
	\sigma_{j,t}=h'_{j,t} +\imath g_{j,t}:\r\to\Agot^*
\end{equation}
holds for $j=1,\ldots, l$ and $t\in [0,1]$,  and at $t=1$ we also have that
\begin{equation}\label{eq:g1}
	\int_0^1 g_{j,1}(x)dx=0, \qquad j=1,\ldots, l.
\end{equation}
Furthermore, given nontrivial arcs $I_j\subset C_j\setminus \{p_*\}$,
we can choose $\sigma_{j,t}$ to be nondegenerate on $I_j$
(see Def.\ \ref{def:nondegenerate}) for all $t\in [0,1]$ and all $j=1,\ldots,l$.

Fix a nowhere vanishing holomorphic $1$-form $\theta$ on $M$ (such $\theta$ exists by the 
Oka-Grauert principle, see Sect.\ \ref{sec:prelim}). We then have that $2\di u_0 = \sigma_0 \theta$,
where $\sigma_0=2\di u_0/\theta\colon M\to \Agot^*$ is a holomorphic map. 
Let $\sigma_t\colon C=\bigcup_{i=1}^{l}C_i \to \Agot^*$ be the smooth map determined by the equations
\begin{equation}\label{eq:sigmat}
	\sigma_t \,\theta|_{C_j} = \sigma_{j,t} d\zeta_j =\left(h'_{j,t} +\imath g_{j,t}\right) d\zeta_j,
	\quad j=1,\ldots,l;\ \ t\in [0,1].
\end{equation}
Here $\sigma_{j,t}$ is given by (\ref{eq:sigma}).
Note that $\sigma_t\in \Ccal^\infty(C,\Agot^*)$ depends smoothly on $t\in [0,1]$.  

Let $\Pcal\colon\Ccal^\infty(C,\Agot^*)\to (\c^3)^l$ denote the period map
which associates to any map $\sigma \in\Ccal^\infty(C,\Agot^*)$ 
the vector $(\Pcal_j(\sigma))_{j=1}^{l}$ with the components
\begin{equation}\label{eq:periodmap}
	\Pcal_j(\sigma) = \int_{C_j} \sigma\theta = 
	\int_0^1 \sigma_j(x) dx \in\c^3,\quad j=1,\ldots,l.
\end{equation}
The $1$-periodic map $\sigma_j\colon \r\to\r^3$ is just the map $\sigma$ 
expressed in the uniformizing coordinate $\zeta_j=x_j+\imath y_j$ along $C_j=\{y_j=0\}$, 
that is, $\sigma\theta = \sigma_j d\zeta_j$ on $C_j$. 
If $\sigma$ corresponds to the differential $2\di u$ 
of some $u\in \CMI(M)$, in the sense that $2\di u=\sigma\theta$ holds on $C$, 
then the real periods of $\sigma$ vanish, $\Re \Pcal(\sigma)=0$,  
while the imaginary periods $\Im \Pcal(\sigma)$ are the flux of $u$:
\[
	\Im \Pcal_j(\sigma) = \Flux_u(C_j)= \int_{C_j} d^c u = \int_0^1 \Im \sigma_j(x)dx \in\r^3,\quad j=1,\ldots,l.
\]

Since the map $\sigma_{j,t}$ is nondegenerate on the segment $I_j\subset C_j$ 
for any $t\in [0,1]$ and $j=1,\ldots,l$, Lemma \ref{lem:deform} and Remark \ref{rem:deform}
furnish a spray of maps $\sigma_{t,w}  \in \Ccal^\infty(C,\Agot^*)$
$(t\in [0,1])$, depending holomorphically on a complex parameter $w$ 
in a ball $W\subset \c^N$ around the origin $0\in \c^N$ for some big $N$, such that
\begin{itemize}
\item $\sigma_{t,0}=\sigma_{t}$ for all $t\in [0,1]$, and 
\item the map $P=(P_1,\ldots P_l) \colon[0,1]\times W \to (\c^3)^l$  with the components 
\begin{equation}\label{eq:P}
	[0,1]\times W \ni (t,w) \longmapsto P_j(t,w) = \int_{C_j} \sigma_{t,w}\theta =
	\int_0^1 \sigma_{j,t,w} (x)dx \in\c^3
\end{equation}
is submersive with respect to the variable $w$ at $w=0$, i.e., the partial differential
\[
	\di_w P(t,w)|_{w=0}\colon\c^N\to (\c^3)^{l}
\]
is surjective  for every $t\in [0,1]$.
\end{itemize}

In the sequel we shall frequently use that $\Agot^*$ is an Oka manifold 
(see \cite[Proposition 4.5]{AF2}), and hence maps $M\to \Agot^*$  from any Stein manifold $M$ 
(in particular, from an open Riemann surface) to $\Agot^*$ satisfy the Runge and the Mergelyan 
approximation theorems  in the absence of topological obstructions.  
The Runge approximation theorem in this setting amounts to the (basic or parametric)
{\em Oka property with approximation for holomorphic maps to Oka manifolds}; 
see \cite[Theorem 5.4.4]{F2011}. 
(An introductory survey of Oka theory can be found in \cite{F2013}.)  
The global Mergelyan approximation theorem 
on suitable subsets of the source Stein manifold follows by combining the local
Mergelyan theorem, which holds for an arbitrary target manifold
(see \cite[Theorem 3.7.2]{F2011}), and the Oka property.

In the case at hand, the  Riemann surface $M$ retracts onto the union of curves
\[
	C=\bigcup_{j=1}^l C_j. 
\]
Since the ball $W\subset\c^N$ is contractible, 
we infer that any continuous map $[0,1]\times W\times C \to \Agot^*$ extends 
to a continuous map  
\[
	[0,1]\times W\times M\to \Agot^*.
\]
Applying the parametric version of Mergelyan's theorem we can approximate the family of maps
$\sigma_{t,w}\colon C\to \Agot^*$ arbitrarily closely in the smooth topology
by a family of holomorphic maps $f_{t,w}\colon M \to \Agot^*$ 
depending holomorphically on $w\in W$ and smoothly on $t\in [0,1]$. 
(The ball $W$ is allowed to shrink around $0\in\c^N$.)
Furthermore, as the initial map $\sigma_{0,0}=2\di u_0/\theta$ is holomorphic
on $M$, the family $\{f_{t,w}\}$ can be chosen such that $f_{0,0}=\sigma_{0,0}$ on $M$.
If the approximation of $\sigma_{t,w}$ by $f_{t,w}$  is close enough for every $t\in[0,1]$ 
and $w\in W$,  it follows from submersivity of the map (\ref{eq:P})  
and the implicit function theorem that there is a smooth map $w=w(t)\in W$ 
$(t\in  [0,1])$ close to $0$ such that $w(0)=0$ and the homotopy of holomorphic maps 
$f_{t,w(t)}\colon M\to \Agot^*$ satisfies
\begin{equation}\label{eq:ftheta}
		\int_{C_j} f_{t,w(t)} \theta = P_j(t,0), \qquad j=1,\ldots,l,\ \ t\in [0,1].
\end{equation}
(Here $P_j$ is defined by (\ref{eq:P}).) By (\ref{eq:sigma}) and (\ref{eq:P}) we have that
\[
	\Re P_j(t,0) = \Re  \int_{C_j} \sigma_{t,0}\theta = 
	\int_0^1  h'_{j,t}(x)dx = 0,\quad j=1,\ldots,l,\ \ t\in [0,1].
\]
Hence it follows from (\ref{eq:ftheta}) that the real part of the holomorphic 
$1$-form $f_{t,w(t)} \theta$   integrates to a  conformal minimal immersion $u_t\in \CMI(M)$ given by
\begin{equation}\label{eq:utp}
	u_t(p) = u_0(p_*) + \int_{p_*}^p \Re( f_{t,w(t)} \theta ),\qquad p\in M,\ \ t\in[0,1].
\end{equation}
For $t=0$ we get the initial immersion $u=u_0$ since $\Re(f_{0,0}\theta)= \Re(2\di u_0) = du_0$. 
In view of (\ref{eq:g1}) we also have $P(1,0)=0$, i.e., the holomorphic 1-form
$f_{1,w(1)} \theta$  with values in $\Agot^*$ has vanishing periods along the curves $C_1,\ldots,C_l$. 
It follows that $u_1$ is the real part of the null holomorphic immersion $F\colon M\to\c^3$ defined by
\[
	F(p) = u_0(p_*) + \int_{p_*}^p f_{1,w(1)} \theta, \qquad p\in M.
\]
This completes the proof of Theorem \ref{th:main} when $H_1(M;\z)$  if finitely generated.

If $M$ is a compact Riemann surface with smooth boundary then the above proof also gives the following
analogue of Theorem \ref{th:main} for conformal minimal immersions $M\to\r^3$ that are 
smooth up to the boundary. 
%
%
%
%
\begin{theorem}\label{th:main-bordered}
Let $M$ be a compact bordered Riemann surface with nonempty smooth boundary $bM$ and let $r\ge 1$. 
For every conformal minimal immersion $u\colon M\to \r^3$ of class $\Cscr^r(M)$ 
there exists a smooth isotopy $u_t\colon M\to\r^3$ $(t\in [0,1])$ of conformal minimal immersions 
of class $\Cscr^r(M)$ such that $u_0=u$ and $u_1=\Re F$  is the real part of a 
holomorphic null curve $F\colon M\to \c^3$ which is smooth up to the boundary.
\end{theorem}


\section{Proof of Theorem \ref{th:main}: the general case} \label{sec:proof}

For an open Riemann surface $M$ of arbitrary topological type we construct 
an isotopy of conformal minimal immersion $\{u_t\}_{t\in [0,1]}\subset \CMI(M)$ 
satisfying Theorem \ref{th:main} by an inductive procedure. 
As before, we assume that the initial immersion $u_0\in \CMI_*(M)$ is nonflat,
and all steps of the proof will be carried out through nonflat immersions.

Pick a smooth strongly subharmonic Morse exhaustion function $\rho\colon M\to\r$.
We can exhaust $M$ by an increasing sequence 
$\emptyset = M_0\subset M_1\subset\cdots\subset \bigcup_{i=0}^\infty M_i=M$
of compact smoothly bounded domains of the form $M_i=\{p\in M\colon \rho(p)\le c_i\}$,
where $c_0<c_1<c_2<\cdots$ is an increasing sequence of regular values of $\rho$
with $\lim_{i\to\infty} c_i =+\infty$. Each domain $M_i$ is a bordered Riemann surface,
possibly disconnected. We may assume that $\rho$ has at most one critical
point $p_i$ in each difference $M_{i+1}\setminus M_i$. It follows that $M_{i}$ is
$\Oscr(M)$-convex and its interior $\mathring M_i$ is Runge in $M$ for every $i\in \z_+$.

We proceed by induction. The initial step is trivial since $M_0=\emptyset$.
Assume inductively that an isotopy $u^i_t\in \CMI_*(M_i)$ $(t\in [0,1])$
satisfying the conclusion of Theorem \ref{th:main} has already been constructed 
over a neighborhood of $M_i$ for some $i\in \z_+$. 
In particular, $u^i_0$ agrees on $M_i$ with the initial immersion 
$u_0$, while $u^i_1=\Re F^i$ is the real part of a null holomorphic immersion
$F^i$ defined on a neighborhood of $M_i$. We will show that $\{u^i_t\}_{t\in [0,1]}$ 
can be approximated arbitrarily closely in the smooth topology on $[0,1]\times M_i$ by an 
isotopy  $\{u^{i+1}_t\}_{t\in [0,1]}$ satisfying the analogous properties over
a neighborhood of $M_{i+1}$. The limit 
\[
	u_t=\lim_{i\to\infty} u^i_t\in \CMI_*(M)
\]
will clearly satisfy Theorem \ref{th:main}.

Let $C_1,\ldots, C_l\subset \mathring M_i$ be closed, oriented, real analytic curves 
whose homology classes form a basis of $H_1(M_i;\z)$ and which satisfy the other 
properties as in Sect.\  \ref{sec:proof1}. Set 
\[
	C=\bigcup_{j=1}^l C_j \subset M_i
\]
and let $\Pcal\colon\Ccal^\infty(C,\Agot^*)\to (\c^3)^l$ denote the 
period map (\ref{eq:periodmap}). Fix a nowhere vanishing holomorphic 
$1$-form $\theta$ on $M$ (see Sect.\ \ref{sec:prelim}) and write
\begin{equation}\label{eq:fit}
	2\di u^i_t =f^i_t \theta\ \ \text{on}\ M_i, \qquad t\in [0,1],
\end{equation}
where $f^i_t\colon M_i\to \Agot^*$ is a holomorphic map depending smoothly on $t\in  [0,1]$. 
(We adopt the convention that a map is holomorphic on a closed
set in a complex manifold if it is holomorphic on an unspecified open
neighborhood of that set.) Note that the map $f_0=2\di u_0/\theta\colon M\to \Agot^*$ 
is defined and holomorphic on all of $M$.

We consider the following two essentially different cases.

\smallskip
(A) The noncritical case:  $\rho$ has no critical value in $[c_i,c_{i+1}]$. 
\vspace{1mm}

(B) The critical case: $\rho$ has a critical point $p_i\in \mathring M_{i+1}\setminus M_i$.
\smallskip

In case (A) there is no change of topology when passing from $M_i$ to $M_{i+1}$.
By \cite[Lemma 5.1]{AF2} (see also Lemma \ref{lem:deform} above) 
there is a spray of maps $f^i_{t,w}  \in \Oscr(M_i,\Agot^*)$
$(t\in [0,1])$, depending holomorphically on a complex parameter $w$ 
in a ball $W\subset \c^N$ around the origin $0\in \c^N$ for some big $N$, satisfying
the following two properties:
\begin{itemize}
\item $f^i_{t,0}=f^i_{t}$ for all $t\in [0,1]$, and 
\vspace{1mm}
\item the map $P=(P_1,\ldots P_l)\colon[0,1]\times W \to (\c^3)^l$  with the components 
\begin{equation}\label{eq:P2}
	[0,1]\times W \ni (t,w) \longmapsto P_j(t,w) = \int_{C_j} f^i_{t,w}\theta \in\c^3
\end{equation}
is submersive with respect to the variable $w$ at $w=0$, i.e., the partial differential
\[
	\di_w P(t,w)|_{w=0}\colon\c^N\to (\c^3)^{l}
\]
is surjective  for every $t\in [0,1]$. 
\end{itemize}
In view of (\ref{eq:fit}) we have
\[
	\Re P_j(t,0)=  \int_{C_j} \Re\left( f^i_{t,0}\theta \right) 
	= \int_{C_j} du^i_t = 0,   \quad j=1,\ldots,l,\ \ t\in [0,1]
\]
and 
\[
         P_j(1,0) = \int_{C_j}  f^i_{1,0}\theta = \int_{C_j} 2\di u^i_1 =0, \quad j=1,\ldots,l
\]
since $u^i_1=\Re F^i$ for a holomorphic null immersion $F^i\colon M_i\to \c^3$.

Since $\Agot^*$ is an Oka surface and $M_i$ is a strong deformation retract of 
$M_{i+1}$, the same argument as in Sect.\ \ref{sec:proof1} (using the Oka principle for maps to $\Agot^*$)
shows that the spray $f^i_{t,w}$ can be approximated
as closely as desired in the smooth topology on $M_i$ by a spray of holomorphic maps
$f^{i+1}_{t,w}\colon M_{i+1}\to \Agot^*$, depending holomorphically
on $w\in W$ (the ball $W$ shrinks a little) and smoothly on $t\in [0,1]$, such that
$f^{i+1}_{0,0}=f_0|_{M_i}$. 
Assuming that the approximation is close enough, the submersivity property 
of the period map $P$  (\ref{eq:P2}) furnishes a smooth map
$w=w(t)\in W$ $(t\in [0,1])$ close to $0$ such that $w(0)=0$ and we have 
for every $j=1,\ldots,l$ that
\[
	 \int_{C_j} \Re \left( f^{i+1}_{t,w(t)} \theta \right)  = 0 \quad (t\in [0,1]),
	 \qquad \int_{C_j}  f^{i+1}_{1,w(1)} \theta =0.
\]
By integrating the family of 1-forms $\Re ( f^{i+1}_{t,w(t)})$ $(t\in [0,1])$ with the
correct choices of initial values at a chosen initial point in each connected component of $M_i$
(as in (\ref{eq:utp})) we obtain a smooth family of conformal minimal immersions 
$u^{i+1}_t\in \CMI_*(M_{i+1})$ which satisfies the induction step.
This completes the discussion of the noncritical case (A).

Consider now the critical case (B), i.e., $\rho$ has a critical point 
$p_i\in \mathring M_{i+1}\setminus M_i$.
By the assumption this is the only critical point of $\rho$ on $M_{i+1}\setminus M_i$
and is a Morse point. Now $M_{i+1}$ admits a strong deformation retraction 
onto $M_i\cup E$ where $E$ is an embedded analytic arc in the complement of $M_i$,
passing through $p_i$, which is attached transversely with both endpoints to $bM_i$. 
There are two possibilities: 
\begin{itemize} 
\item[\rm (a)] $E$ is attached with both endpoints to the same connected
component of $M_i$; 
\item[\rm (b)] the endpoints of $E$ belong to different connected components of $M_i$.
\end{itemize}
Let us begin with Case (a). The arc $E$ completes inside the domain $M_i$ to a  closed  smooth
embedded curve $C_{l+1}\subset M_{i+1}$  which is a new generator of the homology group 
$H_{1}(M_{i+1};\z)$; hence the latter group is generated by the curves $C_1,\ldots, C_{l+1}$.  
By approximation we may assume that $C_{l+1}$ is real analytic.
Let $\zeta=x+\imath y$ be a uniformizing coordinate in an
open annular neighborhood $W_{l+1}\subset M$ of $C_{l+1}$
(see Sect.\ \ref{sec:prelim}). Let $f^i_t\colon U_i\to \Agot^*$ be the isotopy 
of holomorphic maps from the inductive step,  defined on an open  neighborhood $U_i$ of $M_i$ in $M$
and satisfying (\ref{eq:fit}). In analogy with (\ref{eq:sigmat}) we define the isotopy of maps 
$\sigma^i_{t}\colon C_{l+1} \cap U_i \to \Agot^*$ by the equation 
\begin{equation}\label{eq:fit-sigmat}
 	 f^i_{t}\, \theta|_{C_{l+1} \cap U_i} = \sigma_{t} \, d\zeta, \qquad t\in [0,1].
\end{equation}
For $t=0$ the same equation defines the map 
\[
	\sigma_{0}\colon C_{l+1} \to \Agot^*
\]
on all of $C_{l+1}$ since $f^i_{0}=2\di u_0/\theta\colon M\to \Agot^*$ is globally defined on $M$.
Furthermore, we have $\sigma_{0}(x)=h'_{0}(x) +\imath g_{0}(x)$ where $h_{0}(x)= u_0(x)$
is the restricted immersion $u_0|_{C_{l+1}}$ expressed in the uniformizing coordinate along $C_{l+1}$.

Applying Lemma \ref{lem:periods} to $(g_0,h_0)$ we  find an isotopy of smooth $1$-periodic maps 
\begin{equation}\label{eq:sigma-t}
	\sigma_{t}(x)=h'_{t}(x) +\imath g_{t}(x) \in \Agot^*,\quad x\in\r,\ t\in [0,1],
\end{equation}
which agrees with $\sigma_0$ at $t=0$ and satisfies the following conditions:
\begin{itemize}
\item[\rm (i)] $h_t$ is a nonflat immersion for every $t\in[0,1]$, 
\vspace{1mm}
\item[\rm (ii)]  the extended map agrees with the previously defined map on 
the segment in $[0,1]$ representing the arc $C_{l+1}\cap M_{i}= C_{l+1}\setminus E$, and 
\vspace{1mm}
\item[\rm (iii)]  $\int_0^1 g_1(x)dx=0$.
\end{itemize}
We extend the isotopy $f^i_t$ from $M_i$ to $M_i\cup E$ by the equation 
\begin{equation}\label{eq:fi-t}
	f^i_t \theta= \sigma_{t}\, d\zeta, \quad t\in[0,1].
\end{equation}
The maps $f^i_t\colon M_i\cup E\to \Agot^*$ 
are smooth (also in $t\in[0,1]$) and holomorphic on a neighborhood of $M_i$, and we have 
\begin{equation}\label{eq:periods-i}
	\int_{C_{l+1}} \Re\left(f^i_t \theta\right) = 0,\quad t\in[0,1];\qquad 
	 \int_{C_{l+1}} f^i_1 \theta = \int_0^1 g_1(x)dx=0.
\end{equation}

We now complete the induction step as in the case of surfaces with finite topology
treated in Sect.\ \ref{sec:proof1}; let us outline the main steps. 
First we apply \cite[Lemma 5.1]{AF2} to embed the isotopy $f^i_t$ into a spray  
$f^i_{t,w} \colon M_{i}\cup E\to \Agot^*$ of smooth maps which are holomorphic on $M_i$
and depend holomorphically on a parameter $w\in W\subset\c^N$ in a ball of $\c^N$ for some big $N$
such that $f^i_{0,0}=2\di u_0/\theta$ and the period map
$(t,w)\mapsto P(t,w)\in(\c^3)^{l+1}$ with the components
\[
	P_j(t,w) = \int_{C_j} f^i_{t,w}\theta  \in\c^3,\quad j=1,\ldots,l+1
\]
is submersive with respect to $w$ at $w=0$. (Compare with (\ref{eq:P}).
This also follows from the proof of Lemma \ref{lem:deform} and 
Remark \ref{rem:deform} above.) Since $\Agot^*$ is an Oka manifold,
the parametric Mergelyan approximation theorem \cite[Theorem 5.4.4]{F2011}
allows us to approximate the spray $f^i_{t,w}$ in the smooth topology on $M_i\cup E$ by a spray
of holomorphic maps $\tilde f^i_{t,w}\colon M_{i+1} \to \Agot^*$, depending 
smoothly on $t\in [0,1]$ and holomorphically on $w\in W$ 
(the ball $W$ is allowed to shrink a little) 
such that $\tilde f^i_{0,0}=f^i_{0,0}=2\di u_0/\theta$. If the approximation
is sufficiently close then the implicit function theorem furnishes
a smooth map $w\colon [0,1]\to  W\subset \c^N$ close to $0$, with $w(0)=0$, such that 
the isotopy of holomorphic maps 
\[
	f^{i+1}_t =\tilde f^{i}_{t,w(t)}\colon M_{i+1} \to  \Agot^*, \quad t\in [0,1] 
\]
satisfies the following properties:
\begin{itemize}
\item[\rm $(\alpha)$] $f^{i+1}_0=f^{i}_0=2\di u_0/\theta$ on $M_{i+1}$,
\vspace{1mm}
\item[\rm $(\beta)$] $f^{i+1}_t$ approximates $f^{i}_t$ as closely as desired in the smooth topology
on $M_i\cup E$ (uniformly in $t\in [0,1]$), 
\vspace{1mm}
\item[\rm $(\gamma)$] $\int_{C_j} \Re\left( f^{i+1}_t \theta\right) = 0$ for all $j=1,\ldots,l+1$ and $t\in [0,1]$, and
\vspace{1mm}
\item[\rm $(\delta)$] $\int_{C_j} f^{i+1}_1 \theta = 0$ for all $j=1,\ldots,l+1$.
\end{itemize}
Property ($\gamma$) ensures that the real part $\Re(f^{i+1}_t \theta)$
of the holomorphic $1$-form $f^{i+1}_t \theta$ 
integrates to a conformal minimal immersion $u^{i+1}_t\colon M_{i+1}\to \r^3$ depending smoothly
on $t\in[0,1]$.  Property ($\alpha$) shows that $u^{i+1}_0=u^i_0=u_0$ 
with correct choices of constants of integration, and $(\beta)$ implies 
that $u^{i+1}_t$ approximates $u^{i}_t$ in the smooth topology on $M_i$.
Finally, property $(\delta)$ ensures that $u^{i+1}_1=\Re F^{i+1}$, where $F^{i+1}\colon M_{i+1}\to \c^3$
is a holomorphic null curve obtained by integrating the holomorphic $1$-form $f^{i+1}_1 \theta$.
This completes the induction step in Case (a).

In Case (b), when the endpoints of the arc $E$ belong to different connected components
of the domain $M_i$, $E$ does not complete to a closed loop inside $M_i$, and
the inclusion $M_i\hookrightarrow M_{i+1}$ induces an isomorphism 
$H_1(M_i;\z)\cong H_1(M_{i+1};\z)$. Let $C_{l+1}\subset M_i\cup E$ 
be a real analytic arc containing $E$ in its relative interior.
Choose a holomorphic coordinate $\zeta$ on a neighborhood of $C_{l+1}$ in $M$ which maps 
$C_{l+1}$ into the real axis and maps $E$ onto the segment $[0,1]\subset \r\subset \c$. 
Let $f^i_t$ and $\sigma^i_t$ be determined by (\ref{eq:fit}) and (\ref{eq:fit-sigmat}), respectively.   
In the local coordinate $\zeta$, $\sigma_t$ is of the form (\ref{eq:sigma-t}) where $h_t(x)$ and $g_t(x)$ 
are defined for $x$ near the endpoints $0,1$ of $\zeta(E)=[0,1]$. 
Clearly we can extend $h_t$ and $g_t$ smoothly to $[0,1]$ such that conditions (i) and (ii)
(stated just below (\ref{eq:sigma-t})) hold. 
The map $f^i_t$ defined by (\ref{eq:fi-t}) then satisfies the first condition in (\ref{eq:periods-i}),
and the second condition is irrelevant. We now complete the inductive step
exactly  as in Case (a).


\section{h-Runge approximation theorem for conformal minimal immersions} \label{sec:h-Runge}

The proof of Theorem \ref{th:main}, given in Sections \ref{sec:proof1} and  \ref{sec:proof},
depends on the Mergelyan approximation theorem applied to period dominating 
sprays with values in $\Agot^*=\Agot\setminus\{0\}$, where $\Agot$ is the null quadric (\ref{eq:Agot}).
We now present a more conceptual approach to this problem.
Theorem \ref{th:h-Runge} below is a homotopy version of the  Runge-Mergelyan approximation theorem 
for isotopies of conformal minimal immersions, with the additional control of one component function  
which is globally defined on the Riemann surface. This will be used in Sect.\ 
\ref{sec:complete} to prove Theorems \ref{th:main2} and \ref{th:main3}.

We begin by introducing the type of sets that we shall consider for the Mergelyan approximation 
(cf.\ \cite[Def.\ 2.2]{AL1} or \cite[Def.\ 7.1]{AF2}).

%
%
%
%
\begin{definition}
\label{def:admissible}
A compact subset $S$ of an open Riemann surface $M$ is said to be {\em admissible} if $S = K\cup \Gamma$, where 
$K=\bigcup \overline D_i$ is a union of finitely many pairwise disjoint, compact, smoothly bounded domains $\overline D_i$ in $M$ and 
$\Gamma = \bigcup \gamma_j$ is a union of finitely many pairwise disjoint analytic arcs or closed curves that intersect $K$ only 
in their endpoints (or not at all), and such that their intersections with the boundary $bK$ are transverse. 
\end{definition}

An admissible subset $S\subset M$ is $\Oscr(M)$-convex  (also called {\em Runge} in $M$) if and only if the inclusion map 
$S\hookrightarrow M$ induces an injective homomorphism $H_1(S;\z)\hookrightarrow H_1(M;\z)$.

Given an admissible set $S=K\cup \Gamma\subset M$, we denote by $\Ogot(S,\Agot^*)$ the set of all 
smooth maps $S\to\Agot^*$  which are holomorphic on an unspecified open neighborhood of $K$ 
(depending on the map). We denote by $\Ogot_*(S,\Agot^*)$ the subset of $\Ogot(S,\Agot^*)$ consisting of those maps mapping no component of $K$ and no component of $\Gamma$ to a ray on $\Agot^*$.

Fix a nowhere vanishing holomorphic 1-form $\theta$ on $M$. (Such exists by the Oka-Grauert principle, see Sec.\ 
\ref{sec:prelim}; the precise choice of $\theta$ will be unimportant in the sequel.)

The following definition of a conformal minimal immersion of an admissible subset emulates 
the spirit of the concept of {\em marked immersion} \cite{AL1} and provides the natural initial objects for the Mergelyan approximation 
by conformal minimal immersions.

%
%
%
%
\begin{definition}\label{def:generalized}
Let $M$ be an open Riemann surface and let $S=K\cup \Gamma \subset M$ be an admissible subset 
(Def.\ \ref{def:admissible}). A {\em generalized conformal minimal immersion} on $S$ 
is a pair $(u,f\theta)$, where $f\in \Ogot(S,\Agot^*)$
and $u\colon S\to\r^3$ is a smooth map which is a conformal minimal immersion on an open 
neighborhood of $K$, such that 
\begin{itemize}
\item $f\theta = 2\partial u$ on an open neighborhood of $K$, and
\vspace{1mm}
\item for any smooth path $\alpha$  in $M$ parametrizing a connected component of $\Gamma$ we have
$\Re(\alpha^*(f\theta)) = \alpha^*(du)= d(u\circ\alpha)$.
%
%
\end{itemize}
A generalized conformal minimal immersion $(u,f\theta)$ is {\em nonflat} if $u$ is 
nonflat on every connected component of $K$ and also on every curve in $\Gamma$, equivalently, if $f\in\Ogot_*(S,\Agot^*)$.
\end{definition}

We denote by $\GCMI(S)$ the set of all generalized conformal minimal immersions on $S$, 
and by $\GCMI_*(S)$ the subset consisting of nonflat generalized conformal minimal immersions.
We have natural inclusions 
\[
	\CMI(S)\subset\GCMI(S),\quad \CMI_*(S)\subset\GCMI_*(S),
\]
where $\CMI(S)$ is the set of conformal minimal immersions on open neighborhoods of $S$.
If $(u,f\theta) \in \GCMI(S)$ then clearly $u|_K\in\CMI(K)$, $\int_C \Re(f\theta)=0$ 
for every closed curve $C$ on $S$, and $u(x)=u(x_0)+\int_{x_0}^x\Re(f\theta)$ for every pair of points 
$x_0,x$ in the same connected component of $S$.

We say that $(u,f\theta)\in\GCMI(S)$ {\em can be approximated in the $\Cscr^1(S)$ topology} by conformal minimal 
immersions in $\CMI(M)$ if there is a sequence $v_i\in \CMI(M)$ $(i\in\n)$ such that $v_i|_S$ converges to
$u|_S$ in the $\Cscr^1(S)$ topology and $\partial v_i|_S$ converges to 
$f\theta|_S$ in the $\Cscr^0(S)$ topology. 
(The latter condition is a consequence of the first one on $K$, but not on $\Gamma$.)

%
%
%
%
\begin{theorem}[h-Runge approximation theorem for conformal minimal immersions]\label{th:h-Runge}
Let $M$ be an open Riemann surface and let $u\in\CMI_*(M)$ be a nonflat
conformal minimal immersion $M\to\r^3$. Assume that $S=K\cup\Gamma\subset M$ is an $\Oscr(M)$-convex admissible subset (Def.\ \ref{def:admissible}) and  $(u_t,f_t\theta)\in\GCMI_*(S)$ $(t\in [0,1])$ is a smooth isotopy of nonflat generalized conformal minimal immersions on $S$ (Def.\ \ref{def:generalized}) with $u_0=u|_S$ and $f_0\theta=(2\partial u)|_S$. 
Then the family $(u_t,f_t\theta)$ can be approximated arbitrarily closely in the $\Cscr^1(S)$ topology  by a smooth family $\wt u_t\in\CMI_*(M)$ $(t\in [0,1])$ satisfying the following conditions:
\begin{enumerate}[\rm (i)] 
\item $\wt u_0=u$.
\vspace{1mm}
\item $\Flux_{\wt u_t} (C)=\int_C \Im(f_t\theta)$ for every closed curve $C\subset S$ and $t\in [0,1]$. 
(See \eqref{eq:fluxmap}.)
\vspace{1mm}
\item Assume in addition that for every $t\in [0,1]$ the third component function $u_t^3$ of $u_t$ 
extends harmonically to $M$ and the third component $f_t^3$ of $f_t$ extends holomorphically to 
$M$ (hence $2\partial u_t^3 =f_t^3\theta$ on $M$).
Then the family $\wt u_t$ can be chosen to satisfy {\rm (i)}, {\rm (ii)} and also 
\[
	\wt u_t^3=u_t^3\quad \text{for all $t\in [0,1]$}.
\]
\end{enumerate}
\end{theorem}

In the proof of Theorem \ref{th:h-Runge} we shall need the following version of 
the h-Runge approximation property for maps into $\Agot^*$ with the control of the periods.

%
%
%
\begin{lemma}\label{lem:h-Runge}
Let $M$ be an open Riemann surface 
and let $f=(f^1,f^2,f^3)\colon M\to\Agot^*$ be a holomorphic map whose image is not contained in a ray in $\c^3$. 
Assume that $S=K\cup\Gamma\subset M$ is an $\Oscr(M)$-convex admissible subset (Def.\ \ref{def:admissible}).
Then every smooth isotopy 
\[
	f_t=(f_t^1,f_t^2,f_t^3)\in \Ogot_*(S,\Agot^*)\quad (t\in [0,1])
\]
with $f_0=f|_S$ can be approximated arbitrarily closely in $\Cscr^1(S)$ by a smooth family of 
holomorphic maps 
\[
	\wt f_t=(\wt f_t^1,\wt f_t^2,\wt f_t^3)\colon M\to\Agot^*\quad (t\in [0,1])
\]
satisfying the following conditions:
\begin{enumerate}[\rm (i)] 
\item $\wt f_0=f$.
\vspace{1mm}
\item $\int_C \wt f_t \theta = \int_C  f_t\theta$ for every closed curve $C\subset S$ and every $t\in [0,1]$.
\vspace{1mm}
\item If $f_t^3$ extends holomorphically to $M$ for all $t\in [0,1]$, then we can choose the family $\wt f_t$
such that $\wt f_t^3=f_t^3$ for all $t\in [0,1]$.
\end{enumerate}
\end{lemma}

\begin{proof}[Proof of Lemma \ref{lem:h-Runge}]
An isotopy satisfying  properties (i) and (ii) is obtained by following the proof of (the special case of) Theorem \ref{th:main} in Sect.\ \ref{sec:proof1}. Here is a brief sketch.

By Lemma \ref{lem:deform} and Remark \ref{rem:deform}
we can embed the isotopy $\{f_t\}_{t\in [0,1]}$ into a holomorphic period dominating spray of smooth maps 
$f_{t,w}=(f_{t,w}^1,f_{t,w}^2,f_{t,w}^3)\colon S\to\Agot^*$. 
Here, $w$ is a parameter in a ball $W\subset\c^N$ around the origin in a 
complex Euclidean space for some big $N$, $f_{t,w}$ depends holomorphically on $w$ and smoothly
on $t\in[0,1]$, and $f_{t,0}=f_t$ for all $t$. The phrase {\em period dominating} refers to a fixed
finite set of closed loops in $S$ forming a basis of the first homology group
$H_1(S;\z)$.

Since $\Agot^*$ is an Oka manifold, we have the Mergelyan approximation property for maps from Stein manifolds (in particular, from open Riemann surfaces) to $\Agot^*$ in the absence of topological obstructions. (See the argument and the references given in Sect.\ \ref{sec:proof1} above.) In the case at hand, the map $f=f_0\colon M\to \Agot^*$ is globally defined and the domain
$[0,1]\times W$ of the spray $f_{t,w}$ is contractible, so there are no topological 
obstructions to extending these maps continuously to all of $M$. 
Applying the Mergelyan approximation theorem on $S$ we obtain a spray of holomorphic 
maps $\wt f_{t,w}\colon M\to\Agot^*$, depending holomorphically on $w$ 
(whose domain is allowed to shrink a little) and smoothly on $t\in[0,1]$, such that
$\wt f_{t,w}$ approximates $f_{t,w}$ in the $\Cscr^1$ topology on $S$, and $\wt f_{0,0} = f_0$ holds on $M$. 
Within this family  we can then pick  an isotopy $\wt f_t = \wt f_{t,w(t)}$ $(t\in[0,1])$ 
satisfying properties (i) and (ii). The smooth function $[0,1]\ni t\mapsto  w(t)\in\c^N$ 
is  chosen by the implicit function theorem such that $w(0)=0$ (which implies property {\rm (i)}), 
$w(t)$ is close to $0\in \c^N$ for all $t\in[0,1]$ (to guarantee the approximation on $S$),
and $\wt f_t$ satisfies the period conditions in property {\rm (ii)}.

It remains to show that we can also fulfill property {\rm (iii)}.
Let $C_1,\ldots,C_l$ be closed, oriented, analytic curves in $S$ whose homology classes form a basis of $H_1(S;\z)$.
Assume that $f_t^3$ extends holomorphically to $M$ for all $t\in [0,1]$. We argue as in \cite[Theorem 7.7]{AF2}. Set $\Agot'=\Agot\cap\{z_1=1\}$ and observe that $\Agot\setminus\{z_1=0\}$ is biholomorphic to $\Agot'\times\c^*$ (in particular, $\Agot'\times\c^*$ is an Oka manifold), and the projection $\pi_1\colon \Agot\to\c$ is a trivial fiber bundle with Oka fiber $\Agot'$ except over $0\in\c$
where it is ramified. We may embed the isotopy $f_t$ into a spray $f_{t,w}=(f_{t,w}^1,f_{t,w}^2,f_{t,w}^3)\colon S\to\Agot^*$ of smooth maps which are holomorphic on a neighborhood of $K$ and depend holomorphically on a parameter $w\in W\subset\c^N$ in a ball of some $\c^N$
such that $f_{0,0}=f_0$, the third component $f_{t,w}^3$ of $f_{t,w}$ equals $f_t^3$ for all $(t,w)\in [0,1]\times W$, and the period map
$(t,w)\mapsto P(t,w)\in(\c^2)^l$ with the components
\[
	P_j(t,w) = \int_{C_j} (f_{t,w}^1,f_{t,w}^2)\theta  \in\c^2,\quad j=1,\ldots,l,
\]
is submersive with respect to $w$ at $w=0$.   
Up to slightly shrinking the ball $W$, the Oka principle for sections of ramified holomorphic maps with Oka fibers (see \cite{F2003} or \cite[Sec.\ 6.13]{F2011}) 
enables us to approximate the spray $f_{t,w}$ in the smooth topology on $S$ by a spray
of holomorphic maps $\wt  f_{t,w}\colon M \to \Agot^*$, depending 
smoothly on $t\in [0,1]$ and holomorphically on $w\in W$, such that $\wt  f_{0,0}=f$ 
and the third component $\wt  f_{t,w}^3$ of $\wt  f_{t,w}$ equals $f_{t,w}^3=f_t^3$ 
for all $(t,w)\in [0,1]\times W$. If the approximation
is close enough, then the implicit function theorem furnishes
a smooth map $w\colon [0,1]\to  W\subset \c^N$ close to $0$, with $w(0)=0$, such that 
the isotopy of holomorphic maps $\wt  f_t:=\wt  f_{t,w(t)}\colon M\to  \Agot^*$
$(t\in [0,1])$ satisfies {\rm (i)}, {\rm (ii)}, and {\rm (iii)}.
For further details we refer to the proof of \cite[Theorem 7.7]{AF2}.
\end{proof}

Given a compact bordered Riemann surface $\overline R=R\cup bR$
with smooth boundary $bR$ consisting of finitely many smooth Jordan curves, we denote by $\Ascr^r(\overline R)$ 
the set of all maps $\overline R\to\c$ of class $\Cscr^r$ $(r\in\z_+)$ that are holomorphic on the interior $R$ of $\overline R$.

%
%
%
%
\begin{proof}[Proof of Theorem \ref{th:h-Runge}]
Pick a smooth strongly subharmonic Morse exhaustion function $\rho\colon M\to\r$. We exhaust $M$ by an increasing sequence 
\[
	M_1\subset M_2\subset\cdots\subset \bigcup_{j=1}^\infty M_j=M
\]
of compact smoothly bounded domains of the form 
\[
	M_j=\{p\in M\colon \rho(p)\le c_j\},
\]
where $c_1<c_2<\cdots$ is an increasing sequence of regular values of $\rho$
with $\lim_{j\to\infty} c_j =+\infty$. Since $S$ is $\Oscr(M)$-convex, we can choose $\rho$ and $c_1$ such that 
$S\subset \mathring M_1$ and $S$ is a strong deformation retract of $M_1$; 
in particular, the inclusion $S\hookrightarrow M_1$ induces an isomorphism 
\[
	H_1(S;\z)\cong H_1(M;\z)
\]
of their homology groups. 
Each domain $M_j=\mathring M_j\cup b\mathring M_j$ is a compact bordered Riemann surface,
possibly disconnected. We may assume that $\rho$ has at most one critical
point $p_j$ in each difference $M_{j+1}\setminus M_j$. It follows that $M_j$ is
$\Oscr(M)$-convex and $\mathring M_j$ is Runge in $M$ for every $j\in \n$.

We proceed by induction. In the first step we obtain an extension from $S$ to $M_1$.

Assume for simplicity that $M_1$ and so $S$ are connected; the same argument works on any connected component. Pick a point $x_0\in S$. By Lemma \ref{lem:h-Runge}, the family $f_t$ $(t\in[0,1])$ can be approximated arbitrarily closely in $\Cscr^1(S)$ by a smooth isotopy of maps 
\[
	f_{t,1}=(f_{t,1}^1,f_{t,1}^2,f_{t,1}^3)\colon M_1\to\Agot^*
\] 
of class $\Ascr^1(M_1)^3$ such that the family of conformal minimal immersions 
\[
	u_{t,1}=(u_{t,1}^1,u_{t,1}^2,u_{t,1}^3)\in \CMI_*(M_1)
\]
given by
\[
	u_{t,1}(x)=u_t(x_0)+\int_{x_0}^x \Re (f_{t,1}\theta),\quad x\in M_1,
\]
is well defined and satisfies
\begin{enumerate}[\rm (i{$_1$})]
\item $u_{0,1}=u|_{M_1}$,
\vspace{1mm}
\item $\Flux_{u_{t,1}}(C)=\int_C \Im(f_t\theta)$ for every closed curve $C\subset S$ and every $t\in [0,1]$, and
\vspace{1mm}
\item $u_{t,1}^3=u_t^3|_{M_1}$ for all $t\in [0,1]$ provided that the assumptions in Theorem \ref{th:h-Runge} {\rm (iii)} hold.
\end{enumerate}

Assume inductively that for some $j\in \n$ we have already constructed a smooth isotopy 
\[
	u_{t,j}=(u_{t,j}^1,u_{t,j}^2,u_{t,j}^3)\in \CMI_*(M_j),\quad t\in [0,1]
\] 
satisfying 
\begin{enumerate}[\rm (i{$_j$})]
\item $u_{0,j}=u|_{M_j}$,
\item $\Flux_{u_{t,j}}(C)=\int_C \Im(f_t\theta)$ for every closed curve $C\subset S$ and every $t\in [0,1]$, and
\item $u_{t,j}^3=u_t^3|_{M_j}$ for all $t\in [0,1]$ provided that the assumptions in Theorem \ref{th:h-Runge} {\rm (iii)} hold.
\end{enumerate}

Let us show that the smooth isotopy $\{u_{t,j}\}_{t\in [0,1]}$ 
can be approximated arbitrarily closely in the smooth topology on $[0,1]\times M_j$ by a smooth
isotopy  $\{u_{t,j+1}\}_{t\in [0,1]}\subset\CMI_*(M_{j+1})$ satisfying the analogous properties. The limit 
$\wt u_t=\lim_{j\to\infty} u_{t,j}\in \CMI_*(M)$ will clearly satisfy Theorem \ref{th:h-Runge}. Indeed, properties {\rm (i$_j$)}, {\rm (ii$_j$)}, and {\rm (iii$_j$)} trivially imply {\rm (i)}, {\rm (ii)}, and {\rm (iii)}, respectively.

\vspace{1mm}

{\bf The noncritical case:} Assume that $\rho$ has no critical value in $[c_j,c_{j+1}]$. In this case 
$M_j$ is a strong deformation retract of $M_{j+1}$. As above, we finish by using Lemma \ref{lem:h-Runge} applied to the family of maps $f_{t,j}\colon M_j\to\Agot^*$ $(t\in [0,1])$ given by $2\partial u_{t,j}=f_{t,j}\theta$ on $M_j$.

\vspace{1mm}

{\bf The critical case:} Assume that $\rho$ has a critical point $p_{j+1}\in M_{j+1}\setminus M_j$. By the assumptions on $\rho$, $p_{j+1}$ is the only critical point of $\rho$ on $M_{j+1}\setminus M_j$ and is a Morse point. Since $\rho$ is strongly, the Morse index of $p_{j+1}$ is either $0$ or $1$. 

\vspace{1mm}

If the Morse index of $p_{j+1}$ is $0$, then a new (simply connected) component of the sublevel set $\{\rho\leq r\}$ 
appears at $p_{j+1}$ when $r$ passes the value $\rho(p_{j+1})$. In this case 
\[
	M_{j+1}=M_{j+1}'\cup M_{j+1}''
\] 
where $M_{j+1}'\cap M_{j+1}''=\emptyset$, $M_{j+1}''$ is a simply connected component of $M_{j+1}$, 
and $M_j$ is a strong deformation retract of $M_{j+1}'$. Let $\Omega\subset M_{j+1}''$ be a smoothly bounded 
compact disc that will be specified later. It follows that $M_j\cup \Omega$ is a strong deformation retract of $M_{j+1}$. 
Extend $\{u_{t,j}=(u_{t,j}^1,u_{t,j}^2,u_{t,j}^3)\}_{t\in [0,1]}$ to $\Omega$ as any smooth isotopy of conformal minimal 
immersions such that $u_{0,j}|_\Omega=u|_\Omega$; for instance one can simply take $u_{t,j}|_\Omega=u|_\Omega$ 
for all $t\in[0,1]$. If the assumptions in Theorem \ref{th:h-Runge} {\rm (iii)} hold, then take this extension to also satisfy 
$u_{t,j}^3|_\Omega=u_t^3|_\Omega$ for all $t\in [0,1]$. For instance, one can choose $\Omega$ such that $f_t^3$ 
does not vanish anywhere on $\Omega$ for all $t\in[0,1]$, pick $x_0\in \mathring\Omega$, and take
\[
u_{t,j}(x)=y_t+\Re\int_{x_0}^x\left( \frac12\left(\frac1{g}-g\right), \frac{\imath}2\left(\frac1{g}+g\right),1\right)f_t^3\theta,\quad x\in\Omega,
\]
where $y_t=(y_t^1,y_t^2,y_t^3)\in\r^3$ depends smoothly on $t\in[0,1]$ and satisfies $y_0=u(x_0)$ and $y_t^3=u_t^3(x_0)$ 
for all $t\in [0,1]$, and $g$ is the complex Gauss map of $u$ (cf.\ \eqref{eq:gaussmap} and \eqref{eq:WR} and observe that 
$g$ is holomorphic and does not vanish anywhere on $\Omega$). This reduces the proof to the noncritical case.

If the Morse index of $p_{j+1}$ is $1$, then the change of topology of the sublevel set $\{\rho\leq r\}$ at $p_{j+1}$ is described 
by attaching to $M_j$ an analytic arc $\gamma\subset \mathring M_{j+1}\setminus M_j$.
Observe that $M_j\cup \gamma$ is an $\Oscr(M)$-convex strong deformation retract of $M_{j+1}$. Without loss of generality 
we may assume that $M_j\cup \gamma$ is admissible in the sense of Def.\ \ref{def:admissible}. Reasoning as in the critical step 
in Sec.\ \ref{sec:proof}, we extend the family $\{u_{t,j}\}_{t\in [0,1]}$ to a smooth isotopy of nonflat generalized conformal 
minimal immersions $\{(u_{t,j},f_{t,j}\theta)\}_{t\in[0,1]}\subset \GCMI_*(M_j\cup\gamma)$ such that 
\[	
	(u_{0,j}, f_{0,j}\theta)=(u,2\partial u)|_{M_j\cup\gamma}. 
\]
If the assumptions in Theorem \ref{th:h-Runge} {\rm (iii)} hold, 
then we take this extension such that their third components satisfy $u_{t,j}^3=u_t^3|_{M_j\cup\gamma}$ and 
$f_{t,j}^3\theta=(2\partial u_t^3)|_{M_j\cup\gamma}$ for all $t\in [0,1]$. Then, to construct the isotopy 
$u_{t,j+1}\in\CMI_*(M_{j+1})$ $(t\in [0,1])$ meeting {\rm (i{$_{j+1}$})}, {\rm (ii{$_{j+1}$})}, and {\rm (iii{$_{j+1}$})}, 
we reason as in the first step of the inductive process. This finishes the inductive step and proves the theorem.
\end{proof}


\section{Isotopies of complete conformal minimal immersions} \label{sec:complete}

The aim of this section is to prove Theorem \ref{th:main2}. The core of the proof is given by the following technical result.
Recall that given a compact set $K$ in an open Riemann surface $M$, we denote by $\CMI(K)$ the set of maps $K\to\r^3$ extending as conformal minimal immersions to an unspecified open neighborhood of $K$ in $M$, and by $\CMI_*(K)\subset \CMI(K)$ the subset of those immersions which are nonflat on every connected component of $K$.

%
%
%
%
\begin{lemma}\label{lem:JXt}
Let $\overline M=M\cup bM$ be a compact connected bordered Riemann surface. Let $u=(u^1,u^2,u^3)\in\CMI_*(M)$ be a 
conformal minimal immersion which is of class $\Cscr^1(\overline M)$ up to the boundary.  Let $K\subset M$ be an 
$\Oscr(M)$-convex compact set containing the topology of $M$.  Let $u_t=(u_t^1,u_t^2,u_t^3)\in \CMI_*(K)$, $t\in [0,1]$, 
be a smooth isotopy with $u_0=u|_K$.
Assume also that $u_t^3$ extends holomorphically to $M$ for all $t\in [0,1]$.
Let $x_0\in K$ and denote by $\tau$ the positive number given by
\begin{equation}\label{eq:lemmaJX}
\tau:=\dist_{u}(x_0,bM)=\inf\{ \length (u(\gamma))\colon\text{ $\gamma$ an arc in $\overline M$ connecting $x_0$ and $bM$}\}.
\end{equation}
(Here $\length(\cdot)$ denotes the Euclidean length in $\r^3$.)
Then, for any $\delta>0$, the family $u_t$ can be approximated arbitrarily closely in the smooth topology on $K$ by a family 
$\wt u_t\in \CMI_*(M)$ of class $\Cscr^1(\overline M)$, depending smoothly on $t\in [0,1]$ and enjoying the following properties:
\begin{enumerate}[\rm (I)]
\item $\wt u_0=u$.
\vspace{1mm}
\item $\wt u_t^3=u_t^3$ for all $t\in[0,1]$.
\vspace{1mm}
\item $\int_C d^c(\wt u_t-u_t)=0$ for every closed curve $C\subset K$ and every $t\in [0,1]$.
\vspace{1mm}
\item $\dist_{\wt u_t}(x_0,bM)>\tau-\delta$ for all $t\in[0,1]$.
\vspace{1mm}
\item $\dist_{\wt u_1}(x_0,bM)>1/\delta$.
\end{enumerate}
\end{lemma}

\begin{proof}
By Theorem \ref{th:h-Runge} we may assume without loss of generality that $K$ is a compact connected smoothly bounded 
domain in $M$, and the isotopy $u_t$ extends to a smooth isotopy of conformal minimal immersions in $\CMI_*(M)$ of class 
$\Cscr^1(\overline M)$. We emphasize that the latter assumption can be fulfilled while preserving the initial immersion $u$ 
(see Theorem \ref{th:h-Runge} {\rm (i)}) and the third component of any immersion $u_t$ in the family 
(see Theorem \ref{th:h-Runge} {\rm (iii)}). 

Write $u=u_0$. Since $u_t$ depends smoothly on $t$, \eqref{eq:lemmaJX} ensures that, up to possibly enlarging $K$, 
we may also assume the existence of a number $t_0\in ]0,1[$ such that
\begin{equation}\label{eq:taudist}
\dist_{u_t}(x_0,bK)>\tau-\delta/2\quad \text{for all $t\in[0,t_0]$}.
\end{equation}

Let $\theta$ be a nowhere vanishing holomorphic $1$-form of class $\Ascr^1(\overline M)$. Write $2\partial u_t=f_t\theta$ 
where $f_t=(f_t^1,f_t^2,f_t^3)\colon \overline M\to\Agot^*$ is of class $\Ascr^0(\overline M)^3$.

Denote by $m\in\n$ the number of boundary components of $\overline M$. By the assumptions on $K$, the open set 
$M\setminus K$ consists precisely of $m$ connected components $O_1,\ldots, O_m$, each one containing in its 
boundary a component of $bM$. Let $z_j\colon O_j\to\c$ be a conformal parametrization such that $z_j(O_j)$ is a 
round open annulus of radii $0<r_j<R_j<+\infty$ (observe that $O_j$ is a bordered annulus), $j=1,\ldots,m$.

\begin{claim}\label{cla:annuli}
There exist numbers $t_0<t_1<\cdots<t_l=1$, $l\in\n$, and compact annuli $A_{j,k}\subset O_j$, 
$(j,k)\in I:=\{1,\ldots,m\}\times\{1,\ldots,l\}$, satisfying the following properties:
\begin{enumerate}[\rm (i)]
\item $A_{j,k}$ contains the topology of $O_j$ for all $(j,k)\in I$. In particular every arc $\gamma\subset \overline M$ connecting 
$x_0$ and $bM$ contains a sub-arc connecting the two boundary components of $A_{j,k}$ for all $k\in\{1,\ldots,l\}$, 
for some $j\in\{1,\ldots,m\}$.
\vspace{1mm}
\item $z_j(A_{j,k})$ is a round compact annulus of radii $r_{j,k}$ and $R_{j,k}$, where $r_j<r_{j,k}<R_{j,k}<R_j$, for all $(j,k)\in I$.
\vspace{1mm}
\item $A_{j,k}\cap A_{j',k'}=\emptyset$ for all $(j,k)\neq (j',k')\in I$.
\vspace{1mm}
\item $f_t^3$ does not vanish anywhere on $A_{j,k}$ for all $t\in [t_{k-1},t_k]$, for all $(j,k)\in I$.
\end{enumerate}
\end{claim}

\begin{proof}
Let $t\in [t_0,1]$. Since $u_t$ is nonflat, $f_t^3$ does not vanish identically and hence its zeros are isolated on $M$. 
Therefore there exist compact annuli $A_j^t\subset O_j$, $j=1,\ldots,m$, such that
$A_j^t$ contains the topology of $O_j$, $z_j(A_j^t)$ is a round compact annulus, and $f_t^3$ does not vanish anywhere on 
$A_j^t$. Since $f_t^3$ depends smoothly on $t$, there exists an open connected neighborhood $U_t$ of $t$ in $[t_0,1]$ 
such that $f_{t'}^3$ does not vanish anywhere on $A_j^t$ for all $t'\in U_t$. Since $[t_0,1]=\cup_{t\in[t_0,1]} U_t$ is compact, 
there exist numbers $t_0<t_1<\cdots <t_l=1$, $l\in \n$, such that $\cup_{k=1}^l U_{t_k}=[t_0,1]$. Set $A_{j,k}:=A_j^{t_k}$ and 
observe that properties {\rm (i)}, {\rm (ii)}, and {\rm (iv)} hold. To finish  we simply shrink the annuli $A_{j,k}$ in order to 
ensure {\rm (iii)}.
\end{proof}

Since $A_{j,k}\times [t_{k-1},t_k]$ is compact for all $(j,k)$ in the finite set $I$, property {\rm (iv)} gives a small 
number $\epsilon>0$ such that
\begin{equation}\label{eq:epsilon}
	\epsilon < \min \Big\{\Big|\frac{f_t^3\theta}{dz_j}\Big|(x)\colon x\in A_{j,k},\; 
		t\in [t_{k-1},t_k],\; (j,k)\in I\Big\}.
\end{equation}

The next step in the proof of the lemma consists of constructing on each annulus $A_{j,k}$ a Jorge-Xavier type labyrinth 
of compact sets (see \cite{JX} or \cite{Afer,AFL1,AFL2}). 

Let $N$ be a large natural number that will be specified later. 

Assume that $2/N<\min\{R_{j,k}-r_{j,k}\colon (j,k)\in I\}$. For any $n\in\{1,\ldots, 2N^2\}$, we set $s_{j,k;n}:=R_{j,k}-n/N^3$ 
and observe that $r_{j,k}< s_{j,k;n}<R_{j,k}$. We set
\begin{multline}\label{eq:Ljk;n}
	L_{j,k;n}:=\Big\{ x\in A_{j,k}\colon
 		s_{j,k;n}+\frac1{4N^3}\leq |z_j(x)|\leq s_{j,k;n-1}-\frac1{4N^3},\\ 
		\frac1{N^2}\leq \arg((-1)^n z_j(x))\leq 2\pi-\frac1{N^2} \Big\}\subset A_{j,k}.
\end{multline}
By {\rm (iii)}, the compact sets $L_{j,k;n}\subset M\setminus K$ are pairwise disjoint. We also set
\[
	L_{j,k}:=\bigcup_{n=1}^{2N^2} L_{j,k;n},\quad L:=\bigcup_{(j,k)\in I} L_{j,k},
\]
and observe that $K\cap L=\emptyset$ and $K\cup L$ is $\Oscr(M)$-convex. 
(See Fig.\ \ref{fig:Ljk}.)
\begin{figure}[ht]
    \begin{center}
    \scalebox{0.35}{\includegraphics{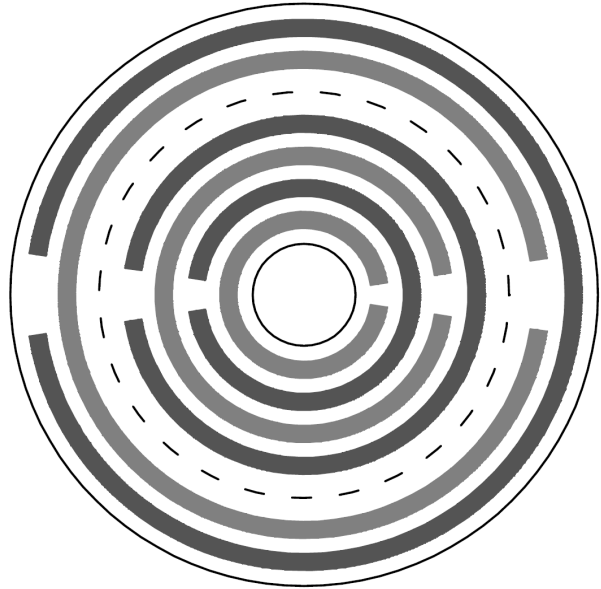}}
        \end{center}
\caption{The labyrinth $z_j(L_{j,k})$ inside the round annulus $z_j(A_{j,k})\subset\c$.}\label{fig:Ljk}
\end{figure}

Denote by $g_t$ the complex Gauss map of $u_t$, i.e. the meromorphic function on $M$
\[
g_t=\frac{f_t^3}{f_t^1-\imath f_t^2},\quad t\in[0,1],
\] 
and recall that
\begin{equation}\label{eq:Weierstrass}
f_t=\left( \frac12\Big(\frac1{g_t}-g_t \Big) , \frac{\imath}2\Big(\frac1{g_t}+g_t \Big) , 1 \right) f_t^3,\quad t\in [0,1]
\end{equation}
(see \eqref{eq:gaussmap} and \eqref{eq:WR}). Since $f_t$ is holomorphic on $M$, {\rm (iv)} ensures that $g_t$ has neither zeros nor poles on $A_{j,k}$ for all $t\in [t_{k-1},t_k]$, $(j,k)\in I$. In particular, since $I$ is finite and $[t_{k-1},t_k]$ is compact, there exists a constant $c_0>0$ such that $|g_t|>c_0$ on $A_{j,k}$ for all $t\in [t_{k-1},t_k]$ and all $(j,k)\in I$. Therefore, we may take a number $\lambda>0$ large enough so that 
\begin{equation}\label{eq:lambda}
|1+\lambda t|\cdotp |g_t|>2N^4\quad \text{on $A_{j,k}$ for all $t\in [t_{k-1},t_k]$ and all $(j,k)\in I$},
\end{equation}
recall that $t_0>0$.

Consider the family of holomorphic maps $h_t=(h_t^1,h_t^2,h_t^3)\colon K\cup L\to\Agot^*$, $t\in [0,1]$, given by
\begin{equation}\label{eq:ft1} 
h_t  =  f_t \quad \text{on $K$},
\end{equation}
\begin{equation}\label{eq:ft2}  
h_t  =  \left( \frac12\Big(\frac1{(1+\lambda t)g_t}-(1+\lambda t)g_t \Big) , \frac{\imath}2\Big(\frac1{(1+\lambda t)g_t}+(1+\lambda t)g_t \Big) , 1 \right) f_t^3 \quad \text{on $L$}.
\end{equation}
The map $h_t$ is said to be obtained from $f_t$ on $L$ by {\em a L\'opez-Ros transformation}; see \cite{LR}.

Since $f_t$ depends smoothly on $t\in [0,1]$, it is clear from \eqref{eq:Weierstrass}, \eqref{eq:ft1}, and \eqref{eq:ft2} that the family $h_t$ depends smoothly on $t\in [0,1]$ as well. Notice that, since $1+\lambda t\neq 0$, the holomorphicity of $f_t$ implies the one of $h_t$, $t\in[0,1]$.  Moreover, equations \eqref{eq:ft1}, \eqref{eq:ft2}, and \eqref{eq:Weierstrass} also give that
\begin{equation}\label{eq:f0}  
h_0=f_0|_{K\cup L},\qquad h_t^3=f_t^3|_{K\cup L}\quad \text{for all $t\in[0,1]$}.
\end{equation}
On the other hand, since $K\cap L=\emptyset$ and $L$ is the union of finitely many 
pairwise disjoint closed discs, \eqref{eq:ft1} ensures that
\begin{equation}\label{eq:ftflux} 
\int_C (h_t-f_t)\theta =0\quad \text{for every closed curve $C\subset K\cup L$ and every $t\in [0,1]$.} 
\end{equation}

In view of \eqref{eq:f0} and \eqref{eq:ftflux}, Lemma \ref{lem:h-Runge} provides a family of holomorphic maps 
\[
	\wt f_t=(\wt f_t^1,\wt f_t^2,\wt f_t^3)\colon M\to\Agot^*, 
\]
depending smoothly on $t\in [0,1]$, such that 
\begin{enumerate}[\rm (i)]
\setcounter{enumi}{4}
\item $\wt f_0=f_0$,
\vspace{1mm}
\item $\wt f_t^3=f_t^3$ for all $t\in [0,1]$,
\vspace{1mm}
\item $\int_C (\wt f_t-f_t)\theta =0$ for every closed curve $C\subset K$ and every $t\in [0,1]$, and 
\vspace{1mm}
\item $\wt f_t$ is as close to $h_t$ in the smooth topology on $K\cup L$ as desired, $t\in [0,1]$.
\end{enumerate}

For each $t\in [0,1]$, consider the conformal minimal immersion $\wt u_t\in \CMI_*(M)$ given by 
\[
\wt u_t(x):=u_t(x_0)+\Re\int_{x_0}^x \wt f_t\theta,\quad x\in M.
\]
Observe that $\wt u_t$ is well defined; see {\rm (vii)} and recall that the periods of $f_t\theta=2\partial u_t$ are purely imaginary. 

Let us check that the family $\{\wt u_t\}_{t\in [0,1]}$ satisfies the conclusion of the lemma. Indeed, since the family $\wt f_t$ 
depends smoothly on $t\in [0,1]$, the same is true for the family $\wt u_t$. Moreover, {\rm (viii)} and \eqref{eq:ft1} ensure that 
$\wt u_t$ can be chosen as close as desired to $u_t$ in the smooth topology on $K$ (uniformly with respect to $t\in [0,1]$); 
recall that $K$ is a compact smoothly bounded domain in $M$, hence arc-connected. On the other hand, we have that
\begin{equation}\label{eq:put}
\wt u_t(x_0)=u_t(x_0),\quad 
     2\, \partial \wt u_t=\wt f_t\theta\quad \text{for all $t\in[0,1]$},
\end{equation}
hence properties {\rm (I)}, {\rm (II)}, and {\rm (III)} directly follow from {\rm (v)}, {\rm (vi)}, and {\rm (vii)}, respectively. 

Let us prove {\rm (IV)} provided that the number $N$ is big enough and the approximation in {\rm (viii)} is close enough. 
Indeed, fix $t\in [0,1]$ and let us distinguish cases. 

First assume that $t\in [0,t_0]$. In this case \eqref{eq:taudist} ensures that 
\[
	\dist_{\wt u_t}(x_0,bM)\geq \dist_{\wt u_t}(x_0,bK)>\tau-\delta
\]
provided that $\wt u_t$ is close enough to $u_t$ on $K$.

Assume now that $t\in [t_0,1]$; hence $t\in [t_{k-1},t_k]$ for some $k\in\{1,\ldots,l\}$. 
Recall that the Riemannian metric $ds_{\wt u_t}^2$ induced on $M$ by the Euclidean metric of $\r^3$ 
via $\wt u_t$ is given by
\begin{equation}\label{eq:ds2ft}
	ds_{\wt u_t}^2=\frac12\, |\wt f_t\theta|^2 \geq |\wt f_t^3\theta|^2
\end{equation}
(see \eqref{eq:metric} and take into account that $x+\frac1{x}\geq 2$ for all $x>0$). In particular, {\rm (viii)} ensures that
\begin{enumerate}[\rm (i)]
\setcounter{enumi}{8}
\item $ds_{\wt u_t}^2$ is as close to $\frac12 |h_t\theta|^2$ 
as desired in the smooth topology on $K\cup L$.
\end{enumerate}
In view of property {\rm (i)} above, it suffices to show that $\length_{\wt u_t}(\gamma)>\max\{\tau-\delta,1/\delta\}$ for any arc $\gamma\subset A_{j,k}$ connecting the two boundary components of the annulus $A_{j,k}$, $j=1,\ldots,m$, where $\length_{\wt u_t}$ denotes the length function in the Riemannian surface $(M,ds_{\wt u_t}^2)$. This will also prove {\rm (V)}. 

Indeed, let $j\in\{1,\ldots,m\}$ and let $\gamma\subset A_{j,k}$ be an arc connecting the two boundary components of $A_{j,k}$. On the one hand, \eqref{eq:ft2} give that
\[
 \frac12 |h_t\theta|^2    = \frac14 \left( \frac1{|1+\lambda t||g_t|}+|1+\lambda t||g_t|\right)^2|f_t^3|^2\, |\theta|^2\quad \text{on $L$}.
\]
This, {\rm (ix)}, \eqref{eq:lambda}, and \eqref{eq:epsilon}, imply that
\begin{equation}\label{eq:est1}
ds_{\wt u_t}^2 > N^8 \epsilon^2 |dz_j|^2 \quad \text{on $L_{j,k}$}.
\end{equation}
On the other hand, \eqref{eq:ds2ft}, {\rm (vi)}, and \eqref{eq:epsilon} give that
\begin{equation}\label{eq:est2}
ds_{\wt u_t}^2 \geq |\wt f_t^3\theta|^2 = |f_t^3\theta|^2>\epsilon^2 |dz_j|^2 \quad \text{on $A_{j,k}$}.
\end{equation} 

The above two estimates ensure that
\begin{equation}\label{eq:est3}
\length_{\wt u_t}(\gamma)>\min\{\frac12,r_{j,k}\}\epsilon N,
\end{equation}
where $r_{j,k}>0$ is the inner radius of $z_j(A_{j,k})$ (see Claim \ref{cla:annuli} {\rm (ii)}).
Indeed, assume first that $\gamma$ {\em crosses} $L_{j,k;n}$, for some $n\in \{1,\ldots,2N^2\}$, in the sense that $\gamma$ contains a subarc $\wh \gamma\subset L_{j,k;n}$ such that $z_j(\wh \gamma)$ connects the two circumferences defining $z_j(L_{j,k;n})$; see \eqref{eq:Ljk;n} and Fig.\ \ref{fig:Ljk}. It follows that the Euclidean length of $z_j(\wh \gamma)$ is at least $1/2N^3$ (cf.\ \eqref{eq:Ljk;n}) and hence \eqref{eq:est1} ensures that $\length_{\wt u_t}(\gamma)>\length_{\wt u_t}(\wh \gamma)>\frac12 \epsilon N$. Assume now that $\gamma$ crosses $L_{j,k;n}$ for no $n\in \{1,\ldots,2N^2\}$. In this case, for any $n\in\{1,\ldots,2N^2-1\}$, $z_j(\gamma)$ surrounds the set $z_j(L_{j,k;n})$ in order to scape by the opening of $z_j(L_{j,k;n+1})$; see \eqref{eq:Ljk;n} and Fig.\ \ref{fig:Ljk}. Since this phenomenon happens at least $2N^2-1$ times, the Euclidean length of $z_j(\gamma)$ is larger than $(2N^2-1)r_{j,k}>Nr_{j,k}$ and hence \eqref{eq:est2} gives that $\length_{\wt u_t}(\gamma)>r_{j,k}\epsilon N$. 

In view of \eqref{eq:est3}, to conclude the proof it suffices to choose $N$ large enough so that $\min\{\frac12,r_{j,k}\}\epsilon N>\max\{\tau-\delta,1/\delta\}$ for all $(j,k)\in I$.
\end{proof}


\begin{proof}[Proof of Theorem \ref{th:main2}]
Let $\rho\colon M\to\r$ be a smooth strongly subharmonic Morse exhaustion function. 
We can exhaust $M$ by an increasing sequence 
\[
	M_0\subset M_1\subset\cdots\subset \bigcup_{i=0}^\infty M_i=M
\]
of compact smoothly bounded domains of the form 
\[
	M_i=\{p\in M\colon \rho(p)\le c_i\},
\]
where $c_0<c_1<c_2<\cdots$ is an increasing sequence of regular values of $\rho$
with $\lim_{i\to\infty} c_i =+\infty$. Each domain $M_i=\mathring M_i\cup bM_i$ is a compact bordered Riemann surface,
possibly disconnected. We may further assume that $\rho$ has at most one critical
point $p_i$ in each difference $M_{i+1}\setminus M_i$. It follows that $M_{i}$ is
$\Oscr(M)$-convex and its interior $\mathring M_i$ is Runge in $M$ for every $i\in \z_+$.

We proceed by induction. Choose a point $x_0\in \mathring M_0$ and set
\begin{equation}\label{eq:taui}
\tau_i:=\dist_{u}(x_0,bM_i)>0\quad \text{for all $i\in\z_+$}.
\end{equation}
The initial step is the smooth isotopy $\{u_t^0:=u|_{M_0}\in\CMI_*(M_0)\}_{t\in [0,1]}$. Assume inductively that we have already constructed for some $i\in\z_+$ a smooth isotopy $u_t^i\in\CMI_*(M_i)$ $(t\in [0,1])$ satisfying the following conditions:
\begin{enumerate}[\rm (a{$_i$})]
\item $u_0^i=u|_{M_i}$.
\vspace{1mm}
\item $\Flux_{u_1^i}(C)=\pgot(C)$ for every closed curve $C\subset M_i$.
\vspace{1mm}
\item $\dist_{u_t^i}(x_0,bM_i)>\tau_i-1/i$ for all $t\in[0,1]$. (This condition is omitted for $i=0$.) 
\vspace{1mm}
\item $\dist_{u_1^i}(x_0,bM_i)>i$.
\end{enumerate}
We will show that $\{u_t^i\}_{t\in [0,1]}$ 
can be approximated arbitrarily closely in the smooth topology on $[0,1]\times M_i$ by an 
isotopy  $\{u^{i+1}_t\}_{t\in [0,1]}$ satisfying the analogous properties over
a neighborhood of $M_{i+1}$. The limit $u_t=\lim_{i\to\infty} u^i_t\in \CMI_*(M)$ will clearly satisfy Theorem \ref{th:main2}. Indeed, properties {\rm (a$_i$)} imply that $u_0=u$, {\rm (b$_i$)} ensure that $\Flux_{u_1}=\pgot$, and {\rm (d$_i$)} give that $u_1$ is complete. Finally, if $u$ is complete, then 
\[
\lim_{i\to\infty} \Big(\tau_i-\frac{1}{i}\Big) =+\infty
\] 
(see \eqref{eq:taui}); hence properties {\rm (c$_i$)} guarantee the completeness of $u_t$ for all $t\in [0,1]$.

Observe that property {\rm (c$_i$)} will not be required in the construction of $u_t^{i+1}$. 
Therefore the construction is consistent with the fact that {\rm (c$_i$)} does not make sense for $i=0$.

\vspace{1mm}

{\bf The noncritical case:} 
Assume that $\rho$ has no critical value in $[c_{i},c_{i+1}]$. In this case
$M_{i}$ is a strong deformation retract of $M_{i+1}$. In view of {\rm (a$_i$)}, {\rm (b$_i$)}, and \eqref{eq:taui}, 
Lemma \ref{lem:JXt} can be applied to the data
\[
	\big( M =M_{i+1} \,,\, u=u|_{M_{i+1}} \,,\, K=M_{i} \,,\, u_t=u_t^{i} \,,\, x_0 \,,\, \tau=\tau_{i+1} \,,\, \delta =1/(i+1) \big),
\]
furnishing a smooth isotopy $u_t^{i+1}\in \CMI_*(M_{i+1})$ which satisfies conditions {\rm (a$_{i+1}$)}--{\rm (d$_{i+1}$)} 
and is as close as desired to $u_t^{i}$ in the smooth topology on $M_{i}$.

\vspace{1mm}

{\em The critical case:} Assume that $\rho$ has a critical point $p_{i+1}\in M_{i+1}\setminus M_{i}$. By the assumptions 
on $\rho$, $p_{i+1}$ is the only critical point of $\rho$ on $M_{i+1}\setminus M_{i}$ and is a Morse point. 
Since $\rho$ is strongly, the Morse index of $p_{i+1}$ is either $0$ or $1$. 

\vspace{1mm}

If the Morse index of $p_{i+1}$ is $0$, then a new (simply connected) component of the sublevel set $\{\rho\leq r\}$ appears 
at $p_{i+1}$ when $r$ passes the value $\rho(p_{i+1})$. In this case 
\[
	M_{i+1}=M_{i+1}'\cup M_{i+1}''
\] 
where $M_{i+1}'\cap M_{i+1}''=\emptyset$, $M_{i+1}''$ is a simply connected component of $M_{i+1}$, and $M_{i}$ 
is a strong deformation retract of $M_{i+1}'$. Extend $u_t^{i}$ by setting $u_t^{i}=u$ on $M_{i+1}''$ for all $t\in[0,1]$. 
This reduces the proof to the noncritical case.

If the Morse index of $p_{i+1}$ is $1$, then the change of topology of the sublevel set $\{\rho\leq r\}$ at $p_{i+1}$ 
is described by attaching to $M_{i}$ an analytic arc $\gamma\subset \mathring M_{i+1}\setminus M_{i}$.
In this case we take $r'\in ]\rho(p_{i+1}),c_{i+1}[$ and set $W=\{\rho\leq r'\}$. By the assumptions, 
we have that $W=\mathring W\cup bW$ is an 
$\Oscr(M)$-convex compact bordered Riemann surface which is a strong deformation retract of $M_{i+1}$. Arguing as in the critical 
case in Sec. \ref{sec:proof} we may approximate $\{u_t^{i}\}_{t\in [0,1]}$ arbitrarily closely in the smooth topology on $[0,1]\times M_i$ 
by an isotopy  $\{\wh u^{i}_t\}_{t\in [0,1]}\subset\CMI_*(W)$ satisfying $\wh u_0^{i}=u|_W$ and $\Flux_{\wh u_1^{i}}(C)=\pgot(C)$ for 
every closed curve $C\subset W$ (take into account {\rm (a$_i$)} and {\rm (b$_i$)}). Further, \eqref{eq:taui} ensures that 
$\dist_{\wh u_0^i}(x_0,bW)>\tau_i$. Again this reduces the construction to the noncritical case and concludes the proof of the theorem.
\end{proof}

In a different direction, we can construct an isotopy of conformal minimal immersions from a given immersion to a 
complete one without changing the flux map.

\begin{theorem}\label{th:main3}
Let $M$ be a connected open Riemann surface of finite topology. For every smooth isotopy $u_t\in\CMI_*(M)$ $(t\in [0,1])$ 
there exists a smooth isotopy $\wt u_t\in\CMI_*(M)$ $(t\in [0,1])$ of conformal minimal immersions such that $\wt u_0=u_0$, 
$\wt u_1$ is complete, the third component of $\wt u_t$ equals the one of $u_t$ for all $t\in [0,1]$, and the flux map of $\wt u_t$ 
equals the one of $u_t$ for all $t\in [0,1]$. Furthermore, if $u_0$ is complete then there exists such an isotopy where 
$\wt u_t$ is complete for all $t\in [0,1]$.
\end{theorem}

\begin{proof}
Let $\rho\colon M\to\r$ be a smooth strongly subharmonic Morse exhaustion function. 
Since $M$ is of finite topology, we can exhaust it by a
sequence 
\[
	M_0\subset M_1\subset\cdots\subset \bigcup_{i=0}^\infty M_i=M
\]
of compact smoothly bounded domains of the form 
\[
	M_i=\{p\in M\colon \rho(p)\le c_i\},
\]
where $c_0<c_1<c_2<\cdots$ is an increasing sequence of regular values of $\rho$
such that $\lim_{i\to\infty} c_i =+\infty$ and there is no critical point of $\rho$ in $M\setminus M_0$. 
Then each domain $M_i=\mathring M_i\cup bM_i$ is a connected compact bordered Riemann surface which is 
an $\Oscr(M)$-convex strong deformation retract of $M_{i+1}$ and of $M$.
Pick $x_0\in \mathring M_0$ and set
\[
	\tau_i=\dist_{u_0}(x_0,bM_i)>0\quad \forall i\in\z_+.
\]

We proceed by induction. The initial step is the isotopy 
\[
	\{u_t^0:=u_t|_{M_0}\}_{t\in[0,1]}.
\]
Assume inductively that we have already constructed for some $i\in\z_+$ a smooth isotopy 
\[
	u_t^i\in\CMI_*(M_i),\quad t\in [0,1]
\]
satisfying the following conditions:
\begin{itemize}
\item $u_0^i=u_0|_{M_i}$.
\vspace{1mm}
\item The third component of $u_t^i$ equals the third component of $u_t$ restricted to $M_i$ for all $t\in[0,1]$.
\vspace{1mm}
\item $\Flux_{u_t^i}(C)=\Flux_{u_t}(C)$ for every closed curve $C\subset M_i$ and all $t\in[0,1]$.
\vspace{1mm}
\item $\dist_{u_t^i}(x_0,bM_i)>\tau_i-1/i$ for all $t\in[0,1]$, $i\in\n$.
\vspace{1mm}
\item $\dist_{u_1^i}(x_0,bM_i)>i$.
\end{itemize}
Reasoning as in the proof of Theorem \ref{th:main2}, Lemma \ref{lem:JXt} ensures that $\{u_t^i\}_{t\in [0,1]}$ can be approximated arbitrarily closely in the smooth topology on $[0,1]\times M_i$ by an isotopy  $\{u^{i+1}_t\}_{t\in [0,1]}$ satisfying the analogous properties over $M_{i+1}$. The limit $\wt u_t=\lim_{i\to\infty} u^i_t\in \CMI_*(M)$ clearly satisfies Theorem \ref{th:main3}.
\end{proof}


\section{On the topology of the space of conformal minimal immersions}
\label{sec:topology}

Theorem \ref{th:main}  amounts to saying that every 
path connected component of $\CMI(M)$ contains a path 
connected component of $\Re\NC(M)$. (See Sec.\ \ref{sec:prelim} for the notation.) 
The proof (see Secs.\ \ref{sec:proof1} and \ref{sec:proof})
shows that a nonflat $u_0\in \CMI_*(M)$ can be connected by a path 
in $\CMI_*(M)$ to some $u_1\in \Re\NC_*(M)$.  The following natural question appears:

\begin{problem}
Are the natural inclusions 
\[
	\Re\NC(M) \hookrightarrow \CMI(M), \quad  
	\Re\NC_*(M) \hookrightarrow \CMI_*(M),
\]
weak (or even strong) homotopy equivalences? 
\end{problem}

In order to show that the inclusion $\iota\colon\Re\NC(M) \hookrightarrow \CMI(M)$ is a 
weak homotopy equivalence (i.e., $\pi_k(\iota)\colon\pi_k(\Re\NC(M)) \stackrel{\cong}{\longrightarrow} \pi_k(\CMI(M))$ 
is an isomorphism of the homotopy groups for each $k=0,1,\ldots$),  it suffices to prove that
$\iota$ satisfies the following:

\vspace{1mm}

{\bf Parametric h-principle}: 
Given a pair of compact Hausdorff spaces $Q',Q$, with $Q' \subset Q$ (it suffices to consider 
Euclidean compacts, or even just finite polyhedra) 
and a continuous map $F\colon Q\to \CMI(M)$ such that $F(Q')\subset  \Re\NC(M)$,
we can deform $F$ through a homotopy $F_t\colon Q\to \CMI(M)$ $(t\in [0,1])$ that is fixed 
on $Q'$ to a map $F_1\colon Q\to \Re\NC(M)$, as illustrated by the following diagram.
\vskip -5mm
\[
	\xymatrix{ 
	  Q'  \ar[d]_{incl}  \ar[r] &    \Re\NC(M) \ar[d]^{\iota} \\
	  Q    \ar[r]^{\!\!\!\!\!\! F}  \ar[ur]^{F_1}	 & \CMI(M) } 
\]
\noindent See \cite{EM,Gromov:book} and \cite[Chapter 5]{F2011} for more details. 

We now describe a connection to the underlying topological questions. 
Fix a nowhere vanishing holomorphic $1$-form $\theta$ on $M$.  
(Such a $1$-form exists by the Oka-Grauert principle, cf.\ Theorem 5.3.1 in \cite[p.\ 190]{F2011}.)
It follows from (\ref{eq:di-u}) that for every $u\in \CMI(M)$ we have
\[
	2\di u= f \theta,
\]
where $f=(f_1,f_2,f_3) \colon M\to \Agot^*$ 
is a holomorphic map satisfying 
\[
	\int_C \Re (f\theta)= \int_C du=0
\]
for any closed curve $C$ in $M$.
Furthermore, we have that $u=\Re F$ for some $F\in \NC(M)$ if and only if 
$\int_C f\theta = 0$ for all closed curves $C$ in $M$.

\begin{problem} \label{pr:HE}
Is the map
\begin{equation}\label{eq:Theta}
	\Theta\colon \CMI(M)\to \Oscr(M,\Agot^*),\quad \Theta(u) = 2\di u/\theta
\end{equation}
a weak homotopy equivalence?  Does it satisfy the parametric h-principle?
\end{problem}

Let $\iota\colon \Oscr(M,\Agot^*) \hookrightarrow \Cscr(M,\Agot^*)$ denote the inclusion 
of the space of  holomorphic maps $M\to\Agot^*$ into the space of continuous maps.
Since $\Agot^*$ is an {\em Oka manifold} \cite[Sect.\ 4]{AF2}, $\iota$ is a weak 
homotopy equivalence \cite[Corollary 5.4.8]{F2011}.
Hence the map 
\[
	\wt \Theta = \iota\circ \Theta\colon \CMI(M)\to \Cscr(M,\Agot^*)
\]
is a weak homotopy equivalence if and only if $\Theta$ is.

By \cite[Theorem 2.6]{AF2} every 
$f_0\in \Cscr(M,\Agot^*)$ can be connected by a path in $\Cscr(M,\Agot^*)$ to a 
holomorphic map $f_1\in \Oscr(M,\Agot^*)$ such that $f_1\theta$ is an exact 
holomorphic $1$-form in $M$; thus $f_1=\Theta(\Re F)$ for some $F\in \NC(M)$.
In particular, we have the following consequence.

\begin{corollary}\label{cor:cor1}
Let $M$ be an open Riemann surface and let $\theta$ be a nonvanishing holomorphic $1$-form 
on $M$. Every connected component of $\Cscr(M,\Agot^*)$ contains 
a map of the form $2\di u/\theta$ where $u\in \CMI(M)$.
\end{corollary}

It is natural to ask how many connected components does $\Cscr(M,\Agot^*)$ have. 
The answer comes from the theory of spin structures on Riemann surfaces;
we refer to the preprint \cite{Kusner} by Kusner and Schmitt. 
Here we give a short self-contained explanation; we wish to thank Jaka Smrekar for his help at this point.

Denote the coordinates on $\c^3$ by $z=\xi + \imath \eta$, with $\xi,\eta\in\r^3$, and let 
\[
	\pi\colon \c^3=\r^3\oplus \imath \r^3 \to\r^3
\]
be the projection $\pi(\xi +\imath \eta)=\xi$. Then $\pi\colon \Agot^*\to \r^3\setminus \{0\}$ 
is a real analytic fiber bundle with circular fibers
\begin{equation}\label{eq:fiberbundle2}
	 \Agot \cap \pi^{-1}(\xi) = \{\xi+\imath \eta\in\c^3: \xi\cdotp \eta=0,\ |\xi|=|\eta|\}
	 \cong \mathbb{S}^1, \quad \xi \in \r^3\setminus \{0\}.
\end{equation}
Let $\mathbb{S}^2=\{\xi\in \r^3\colon|\xi|=1\}$,  the unit sphere of $\r^3$.
Then $\Agot^*$ is homotopy equivalent to $\Agot^* \cap \pi^{-1}(\mathbb{S}^2)$,
and by (\ref{eq:fiberbundle2}) this is the unit circle bundle of the tangent bundle 
of $\mathbb{S}^2$:  
\[
	\Agot^*\cap \pi^{-1}(\mathbb{S}^2) = S(T\mathbb{S}^2)\cong SO(3).
\] 
It follows that    
\begin{equation}\label{eq:pi1A}
	\pi_1(\Agot^*)\cong \pi_1(SO(3)) \cong \z_2:=\z/2\z.
\end{equation}

An open Riemann surface $M$ has the homotopy type of a finite or countable 
wedge of circles, one for each generator of $H_1(M;\z)$. Fix a pair of points $p\in M$, $q\in \Agot^*$, and let 
$\Cscr_*(M,\Agot^*)$ denote the space of all continuous maps sending $p$ to $q$. It is easily seen that 
$\Cscr(M,\Agot^*)\cong \Cscr_*(M,\Agot^*)\times \Agot^*$. The space $\Cscr_*(M,\Agot^*)$ is homotopy equivalent
to the cartesian product of loop spaces $\Omega(\Agot^*) =\Cscr_*(\mathbb{S}^1,\Agot^*)$,
one for each generator of $H_1(M;\z)$. Since the connected components of 
$\Omega(\Agot^*)$ coincide with the elements of the fundamental group 
$\pi_1(\Agot^*)\cong\z_2$ (see (\ref{eq:pi1A})) and $\Agot^*$ is connected, we infer the following.

%
%
%
%
\begin{proposition} \label{prop:components}
If $M$ is an open Riemann surface and $H_1(M;\z)\cong \z^l$ $(l\in \z_+\cup\{\infty\})$ then 
the connected components of $\Cscr(M,\Agot^*)$ are in one-to-one correspondence with 
the elements of the free abelian group $(\z_2)^l$. Hence each of the spaces $\NC(M)$ and $\CMI(M)$ 
has at least $2^l$ connected components.
\end{proposition} 

The last statement  follows from Theorem \ref{th:main} and Corollary \ref{cor:cor1}.


\subsection*{Acknowledgements}
A.\ Alarc\'{o}n is supported by the Ram\'on y Cajal program of the Spanish Ministry of Economy and Competitiveness, and is partially supported by the MINECO/FEDER grants MTM2011-22547 and MTM2014-52368-P, Spain. 
F.\ Forstneri\v c is supported in part by the research program P1-0291 and the grant J1-5432 from ARRS, Republic of Slovenia. The authors wish to thank Jaka Smrekar for his help with the topological matters in Sect.\ \ref{sec:topology}.
We also thank the referee for useful suggestions which lead to improved presentation.


\vfill\eject

\noindent Antonio Alarc\'{o}n

\noindent Departamento de Geometr\'{\i}a y Topolog\'{\i}a and Instituto de Matem\'aticas IEMath-GR, Universidad de Granada, E--18071 Granada, Spain.

\noindent  e-mail: {\tt alarcon@ugr.es}

\vspace*{0.5cm}

\noindent Franc Forstneri\v c

\noindent Faculty of Mathematics and Physics, University of Ljubljana, and Institute
of Mathematics, Physics and Mechanics, Jadranska 19, SI--1000 Ljubljana, Slovenia.

\noindent e-mail: {\tt franc.forstneric@fmf.uni-lj.si}
\end{document}